\documentclass[11pt, letterpaper]{article}  %amsart
\usepackage{geometry}                % See geometry.pdf to learn the layout options. There are lots.
\usepackage{graphicx}
\usepackage{amssymb, amsmath, amsthm, dsfont}
\usepackage{epstopdf}
\usepackage{empheq}
\newcommand{\R}{\mathbb{R}}
\newcommand{\D}{\mathbb{D}}
\newcommand{\EE}{\mathbb{E}}
\newcommand{\Z}{\mathbb{Z}}
\newcommand{\XX}{\mathbb{X}}
\newcommand{\HH}{\mathfrak{H}}
\newcommand{\dom}{\rm Dom}
\usepackage{color}
\newtheorem{theorem}{Theorem}[section]

\newtheorem{thm}[theorem]{Theorem}

\newtheorem{prop}[theorem]{Proposition}
\newtheorem{cor}[theorem]{Corollary}

\newtheorem{lemma}[theorem]{Lemma}

\newtheorem{remark}[theorem]{Remark}

\numberwithin{theorem}{section}
\numberwithin{equation}{section}

\begin{document}
	
\title{Total variation estimates in  the Breuer-Major theorem}
\date{}
\author{David Nualart\thanks{David Nualart was supported by the NSF grant  DMS 1512891}  \,  and
Hongjuan Zhou}
\maketitle

%============================================
\begin{abstract}
	This paper  provides estimates for  the convergence rate  of the total variation distance in the framework of the  Breuer-Major theorem, assuming  some smoothness properties of the  underlying function. The results are proved by applying new bounds for
	the total variation distance between a random variable  expressed as a divergence and a standard Gaussian random variable, which are derived 	 by a combination of techniques of  Malliavin calculus and Stein's method.  The representation
	of a functional of a Gaussian sequence  as a divergence is established by introducing a shift operator on the expansion in Hermite polynomials.
	 Some applications to the asymptotic behavior of power variations of the fractional Brownian motions and to the  estimation of the
Hurst parameter using power variations are presented.
	
\end{abstract}
%============================================


%============================================
\section{Introduction}
%============================================

Consider a centered stationary Gaussian  family of random variables   $X=\{ X_n , n\in \Z\}$  with unit variance.
  For all $k \in \Z$, set $\rho(k) = \EE(X_0 X_k)$, so $\rho(0)=1$ and $\rho(k) = \rho(-k)$.  
   We say that a function $g \in L^2( \mathbb{R}, \gamma)$, where $\gamma$ is the standard Gaussian measure, has  
  {\it Hermite rank}   $d\ge 1$ if
	\begin{equation}\label{hexp}
g(x)= \sum_{m=d} ^\infty c_m  H_m(x),
\end{equation}
 where $c_d \not =0$ and $H_m$ is the $m$th Hermite polynomial. We will  make use of    the following condition that relates the covariance function $\rho$ to the Hermite rank of a function $g$:
 \begin{equation} \label{h1}
\sum_{j \in \Z} |\rho(j)|^d < \infty.
\end{equation}
   Let us  recall the celebrated Breuer-Major theorem for functionals of the stationary Gaussian sequence $X$ (see \cite{bm}). 
  \begin{thm}[Breuer-Major theorem]\label{bm}
  Consider a centered stationary Gaussian  family of random variables   $X=\{ X_n , n\in \Z\}$  with unit variance and covariance function $\rho$.
	Let $g \in L^2( \mathbb{R}, \gamma)$ be a function with Hermite rank $d\ge 1$ and expansion (\ref{hexp}). Suppose that  (\ref{h1})
	holds true.
 Set 
	\begin{equation}\label{bm.sig}
		\sigma^2 = \sum_{m=d}^\infty m! c_m^2 \sum_{k \in \Z} \rho(k)^m.
	\end{equation}
Then the sequence
	  \begin{equation}\label{yn}
	  Y_n := \frac{1}{\sqrt{n}} \sum_{j=1}^n g(X_j) 
	  \end{equation}
  converges in law to the  normal distribution $N(0, \sigma^2)$.
   \end{thm}
   
   The purpose of this paper is to show that, under suitable regularity assumptions on the function $g$, the sequence $Y_n /\sigma_n$,
   where $\sigma_n^2 = \EE(Y_n^2)$,  converges in the total variation distance to the standard normal law $N(0,1)$, and we can estimate the rate of convergence in terms of the covariance function $\rho$. To show these results we will apply a combination of Stein's method for normal approximations and techniques of Malliavin calculus.  The combination of Stein's method with Malliavin calculus to study normal approximations was first developed by Nourdin and Peccati  (see the pioneering work \cite{np-ptrf} and the monograph  \cite{np-book}).  For random variables on a fixed Wiener chaos, these techniques provide a quantitative version of the {\it Fourth Moment Theorem} proved by Nualart and Peccati in \cite{nunugio}. A basic result in this direction is the following proposition.
   Along the paper $Z$ will denote a $N(0,1)$  random variable.

   \begin{prop}\label{bd.4m-1}
	Let $F$ be a random variable in  the $q$th ($q \geq 2$) Wiener chaos with unit variance.  Then
	\begin{equation} \label{equ80}
d_{\rm TV}(F, Z) \leq 2 \sqrt{{\rm Var} \left(\frac{1}{q}\|DF\|_{\mathfrak{H}}^2\right)} \leq 2\sqrt{\frac{q-1}{3q}(\EE(F^4)-3)}\,,
\end{equation}
where $D$ denotes the derivative in the sense of Malliavin calculus and $d_{\rm TV}$ is the total variation distance.
\end{prop}
In the context of the Breuer-Major theorem, this result can be applied to obtain a rate of convergence for the total variation distance 
$d_{\rm TV} (Y_n/ \sigma_n, Z)$, provided $g=H_d$  and condition (\ref{h1}) holds (see \cite{np-ptrf}). 
 Later on, the rate  of convergence was  improved  in \cite{bbl} using an approach based on  the spectral density.
 
   In the reference  \cite{np-15},  with an intensive  application of  Stein's method combined with Malliavin calculus, Nourdin and Peccati improved
  the estimate (\ref{equ80}), obtaining  the following matching upper and lower bounds for the total variation distance.
  
  \begin{prop}\label{bd.4m-2}
	 Let $F$ be a  random variable in  the $q$th ($q \geq 2$) Wiener chaos with unit variance.   Then, there exist constants $C_1, C_2>0$, depending on $q$,  such that  
	    \[
	C_1 \max\{ |\EE(F^3)|, \EE(F^4)-3\}    \le d_{\rm TV}(F, Z) \leq C_2 \max\{ |\EE(F^3)|, \EE(F^4)-3\} \,.
	\]
\end{prop}
In the paper \cite{bbnp}, it is proved that that $|\EE(F^3)| \leq C \sqrt{\EE(F^4)-3}$, which  trivially indicates that  the bound in Proposition \ref{bd.4m-2} is better than (\ref{equ80}). Furthermore,    using an analytic characterization of cumulants and  Edgeworth-type expansions,  the authors of \cite{bbnp}  proved that,  for a normalized sequence $F_n$ which belongs to the $q$th Wiener chaos and converges to $Z$ in distribution as $n \to \infty$, the rate of convergence of the total variation distance is characterized by the third and fourth cumulants.  

The literature on the rate of convergence for normal approximations is focused on random variables on a fixed Wiener chaos.
The goal of this paper is to provide an answer to the following question:

\medskip
\noindent
{\bf Question:} To what extent  Propositions \ref{bd.4m-1} and \ref{bd.4m-2} can be generalized  to random variables that are not in a fixed chaos and
how this approach is applied in the context of the Breuer Major theorem?

\medskip
We cannot expect that, in this more general framework,  the convergence to a normal distribution is characterized by the third and fourth cumulants, and new functionals will appear. In the first part of the paper, we  consider random variables that can be written as  divergences, that is $F=\delta(u)$, where $\delta $ is the adjoint of the derivative operator in the Malliavin calculus. We will use Stein's method and Malliavin calculus   to provide three different bounds (see Propositions \ref{stein.hd0}, \ref{stein.hd} and  \ref{stein.hd2}) for  $d_{\rm TV} (F,Z)$.
If $F$ is in some fixed chaos, the bound in Proposition \ref{stein.hd0} should be the same as  that of Proposition \ref{bd.4m-1}  and the bound in Proposition \ref{stein.hd} should coincide with  that of Proposition \ref{bd.4m-2}.  Actually, the proof of Proposition \ref{stein.hd} has been inspired by the approach used to derive the upper bound in Proposition \ref{bd.4m-2}.

The second part of the paper is devoted to derive  upper bounds  for the total variation distance in the context of  the Breuer-Major theorem, applying  the estimates provided by
   Propositions \ref{stein.hd0}, \ref{stein.hd} and \ref{stein.hd2}. To do this, we need to represent $g(X_j)$ as a divergence $\delta(u)$. A basic ingredient for this representation is the  shift operator $T_1$ (see formula (\ref{t1a}) below) defined
using the expansion of $g$ into a series of Hermite polynomials. It turns out that the representation obtained through $T_1$ coincides with the classical representation $F= \delta (-DL^{-1} F)$, introduced in  \cite{nuaort}, that plays a fundamental role in    normal approximations by Stein's method and Malliavin calculus.   
The representation of $g(X_j)$ as a divergence (or an iterated divergence) allows us to apply the  integration by parts in the context of Malliavin calculus (or duality between the derivative and divergence operators), which leads to estimates of the expectation of products of  random variables of the form $g^{(k)}(X_j)$.  For this approach to work, we are going  to  assume that the function $g$ belongs to the Sobolev space $\D^{k,p}(\R,\gamma)$, for some $k$ and $p$, of functions that have $k$ weak derivatives with  moments of order $p$ with respect to $\gamma$. 
  
 In this way we have been able to obtain the following results  in the framework of Theorem \ref{bm},  for functions of  Hermite rank one or two.
 \begin{itemize}
 \item[(i)] For  functions $g$ of Hermite rank $d=1$,  assuming $g\in \D^{2,4} (\R,\gamma)$, we have      (see Theorem  \ref{main.d1} below)
 \[
       d_{\rm TV}(Y_n / \sigma_n, Z) \leq C n^{-\frac{1}{2}}.
       \]
       \item[(ii)]  For  functions $g$ of Hermite rank $d=2$,  assuming $g\in \D^{6,8} (\R,\gamma)$,  we have  (see Theorem  \ref{main.d12} below)
    \begin{equation}  \label{rate1}
			   d_{\rm TV}(Y_n / \sigma_n, Z) \leq    
			   C n^{-\frac{1}{2}} \left(\sum_{|k| \leq n} |\rho(k)|^{\frac{3}{2}}\right)^{2}.
			   \end{equation}
 \end{itemize}
    It is worth noticing that the upper bound  (\ref{rate1}) coincides with the optimal rate for the Hermite polynomial $g(x)= x^2-1$  obtained in \cite{bbnp}. Furthermore, in Theorem  \ref{main.d12}, rates worse than (\ref{rate1}) are established under less smoothness on the function $g$.
    
For functions  $g$ of Hermite rank $d\ge 3$ and assuming $g\in \D^{3d-2,4} (\R,\gamma)$, we have established in Theorem \ref{thm3.4} an upper bound for   the total variation  distance $d_{\rm TV}(Y_n / \sigma_n, Z)$   based on  Proposition \ref{stein.hd0},   
   which is a   slight modification  of the rate  derived for the Hermite polynomial $H_d$.  Due to the complexity of the computations, the application of Proposition \ref{stein.hd}  in the case $d\ge 3$ has  not been considered in this paper.

The paper is organized as follows. Section 2 contains some preliminaries on  Malliavin calculus and  Stein's method,  including the definition and properties of the shift operator $T_1$. 
In Section 3,  we derive the three basic estimates for the total variation distance between a divergence $\delta(u)$ and a $N(0,1)$ random variable.  Section 4 contains the main results of the paper. First we  thoroughly analyze the cases $d=1$ and $d=2$ and
establish bounds for the total variation distance in the framework of   the Breuer-Major  theorem and later we consider the case $d\ge 3$, applying
Proposition \ref{stein.hd0}.

As an application, in Section  5  we give the convergence rates for the fractional Gaussian case.  We also discuss some applications to the asymptotic behavior of power variations of the fractional Brownian motions and to the consistency of the estimator of the
Hurst parameter using power variations. The Appendix contains some  technical lemmas used in the proof of the main results and
some inequalities, obtained as an application of the rank-one Brascamp-Lieb inequality and  H\"older's  inequality, which play an important role in the proofs.
 
 %============================================
\section{Preliminaries}
%============================================

In this section, we briefly recall some notions of  Malliavin calculus,  Stein's method and  the Brascamp-Lieb inequality. The shift operator $T_1$ mentioned above is also introduced here.

\subsection{Gaussian analysis}

 Let $\mathfrak{H}$ be a real separable Hilbert space. For any integer $m \geq 1$, we use $\mathfrak{H}^{\otimes m}$ and $\mathfrak{H}^{\odot m}$ to denote the $m$-th tensor product and the $m$-th symmetric tensor product of $\mathfrak{H}$, respectively. Let $\XX = \{\XX(\phi): \phi \in \mathfrak{H}\}$  denote an isonormal Gaussian process over the Hilbert space $\mathfrak{H}$. That means, $\XX$ is a centered Gaussian family of random variables, defined on some probability space $(\Omega, \mathcal{F}, P)$, with covariance $$\EE\left(\XX(\phi)\XX(\psi)\right) = \langle \phi, \psi \rangle_{\mathfrak{H}}, \qquad \phi, \psi \in \mathfrak{H}.$$
We assume that $\mathcal{F}$ is generated by $\XX$.

We denote by $\mathcal{H}_m$ the closed linear subspace of $L^2(\Omega)$ generated by the random variables $\{H_m(\XX(\varphi)): \varphi \in \mathfrak{H}, \|\varphi\|_{\mathfrak{H}}=1\}$, where $H_m$ is the $m$-th Hermite polynomial defined by 
\[
H_m(x)=(-1)^me^{\frac{x^2}{2}}\frac{d^m}{dx^m}e^{-\frac{x^2}{2}},\quad m \geq 1,
\]
and $H_0(x)=1$. The space $\mathcal{H}_m$ is called the Wiener chaos of order $m$. The $m$-th multiple integral of $\phi^{\otimes m} \in \mathfrak{H}^{\odot m}$ is defined by the identity %$I_m(\varphi) = \delta^m (\varphi)$, and in particular,  
$ I_m(\phi^{\otimes m}) = H_m(\XX(\phi))$ for any $\phi\in \mathfrak{H}$. The map $I_m$ provides a linear isometry between $\mathfrak{H}^{\odot m}$ (equipped with the norm $\sqrt{m!}\|\cdot\|_{\mathfrak{H}^{\otimes m}}$) and $\mathcal{H}_m$ (equipped with $L^2(\Omega)$ norm). By convention, $\mathcal{H}_0 = \mathbb{R}$ and $I_0(x)=x$.

The space $L^2(\Omega)$ can be decomposed into the infinite orthogonal sum of the spaces $\mathcal{H}_m$, which is known as the Wiener chaos expansion. Thus, any square integrable random variable $F \in L^2(\Omega)$ has the following expansion,
\[
  F = \sum_{m=0}^{\infty} I_m (f_m) ,
\]
where $f_0 = \mathbb{E}(F)$, and $f_m \in \mathfrak{H}^{\odot m}$ are uniquely determined by $F$. We denote by $J_m$ the orthogonal projection onto  the $m$-th Wiener chaos $\mathcal{H}_m$. This means that  $I_m (f_m) = J_m (F)$ for every $m \geq 0$.

\subsection{Malliavin calculus}   
In this subsection we present  some background of Malliavin calculus with respect to an   isonormal Gaussian process $\XX$. 
We refer the reader to \cite{np-book,nualartbook} for a detailed account on this topic. For a smooth and cylindrical random variable $F= f(\XX(\varphi_1), \dots , \XX(\varphi_n))$, with $\varphi_i \in \mathfrak{H}$ and $f \in C_b^{\infty}(\mathbb{R}^n)$ ($f$ and its partial derivatives are bounded), we define its Malliavin derivative as the $\mathfrak{H}$-valued random variable given by
\[
 DF = \sum_{i=1}^n \frac{\partial f}{\partial x_i} (\XX(\varphi_1), \dots, \XX(\varphi_n))\varphi_i.
\]
By iteration, one can define the $k$-th derivative $D^k F$  as an element of $L^2(\Omega; \mathfrak{H}^{\otimes k})$. For any natural number $k$ and any real number $ p \geq 1$, we define  the Sobolev space $\mathbb{D}^{k,p}$  as the closure of the space of smooth and cylindrical random variables with respect to the norm $\|\cdot\|_{k,p}$ defined by 
\[
 \|F\|^p_{k,p} = \mathbb{E}(|F|^p) + \sum_{i=1}^k \mathbb{E}(\|D^i F\|^p_{\mathfrak{H}^{\otimes i}}).
\]
The divergence operator $\delta$ is defined as the adjoint of the derivative operator $D$ in the following manner.  An element $u \in L^2(\Omega; \mathfrak{H})$ belongs to the domain of $\delta$, denoted by $\dom\, \delta$, if there is a constant $c_u$ depending on $u$ such that 
\[
|\mathbb{E} (\langle DF, u \rangle_{\mathfrak{H}})| \leq c_u \|F\|_{L^2(\Omega)}
\] for any $F \in \mathbb{D}^{1,2}$.  If $u \in \dom \,\delta$, then the random variable $\delta(u)$ is defined by the duality relationship 
\begin{equation} \label{dua}
\mathbb{E}(F\delta(u)) = \mathbb{E} (\langle DF, u \rangle_{\mathfrak{H}}) \, ,
\end{equation}
which holds for any $F \in \mathbb{D}^{1,2}$.  
%If $u=\{u_t, t\in [0,T]\}$ is a stochastic process, whose trajectories belong to $\mathfrak{H}$ almost surely (with the convention $u_t=0$ if  $t\not\in [0,T]$) and $u\in \dom \, \delta$, we make use of the notation $ \int_0^T u_t d\XX_t = \delta(u)$ and call $\delta(u)$ the divergence integral of $u$ with respect to the Gaussian process $\XX$ on $[0,T]$.
%It is worth noting that the  divergence integral of fBm with respect to itself does not exist if $H \in (0,\frac{1}{4})$ because the paths of  the fBm are too irregular (see \cite{cn}). For this reason, in \cite{cn} the authors introduce an extended divergence integral $\delta^*$ such that $\dom \,\delta^* \cap  L^2 (\Omega; \mathfrak{H}) = \dom\, \delta$ and the extended divergence operator $\delta^*$ restricted to $\dom\, \delta$  coincides with the  divergence operator.
In a similar way we can introduce the iterated divergence operator $\delta^k$ for each integer $k\ge 2$, defined by the duality relationship 
\begin{equation} \label{dua2}
\mathbb{E}(F\delta^k(u)) = \mathbb{E}  \left(\langle D^kF, u \rangle_{\mathfrak{H}^{\otimes k}} \right) \, ,
\end{equation}
for any $F \in \mathbb{D}^{k,2}$, where $u\in  {\rm Dom}\, \delta^k \subset L^2(\Omega; \mathfrak{H}^{\otimes k})$.

The Ornstein-Uhlenbeck semigroup $(P_t)_{t \geq 0}$ is the  semigroup of operators on $L^2(\Omega)$ defined by
\[
P_t F = \sum_{m=0}^\infty e^{-mt} I_m(f_m),
\]
if $F$ admits the Wiener chaos expansion $F = \sum_{m=0}^\infty I_m(f_m)$. Denote by $L = \frac{d}{dt}|_{t=0} P_t$ the infinitesimal generator of $(P_t)_{t \geq 0}$ in $L^2(\Omega)$. Then we have $LF=-\sum_{m=1}^\infty m J_m(F)$ for any $F \in {\rm Dom}\, L=\mathbb{D}^{2,2} $. We define the pseudo-inverse of $L$ as $L^{-1} F = -\sum_{m=1}^\infty \frac{1}{m} J_m F$.   We recall the following formula for any centered and square integrable random variable $F$,
\begin{equation} \label{equ5}
L^{-1} F = -\int_0 ^\infty P_t F dt.
\end{equation}
The basic  operators $D$, $\delta$ and $L$ satisfy the relation $ LF =-\delta DF $, for any random variable $F\in\mathbb{D}^{2,2}$.
As a consequence, any centered random variable $F\in L^2(\Omega)$ can be  expressed as a divergence:
\begin{equation}  \label{deltad}
F= \delta (-DL^{-1}F).
\end{equation}
This representation has intensively been used in normal approximations (see \cite{nuaort,nunugio}).
%We say that a random variable $F=\sum_{m=0} ^\infty I_m(f_m)$ has Hermite rank larger than or equal to $d$ if   $ f_0 = \cdots =f_{d-1}$.
 %This is equivalent to say that $F$ is orthogonal to the first $d$ Wiener chaos or, if $F \in \mathbb{D}^{d,2}$,  that $\EE(D^{i} F)=0$ for $i=0,1, \dots, d-1$.

 We denote by $\gamma$ the standard Gaussian measure on $\R$.  The Hermite polynomials $\{H_m(x), m\ge 0\}$  form a complete orthonormal system in $L^2(\R,\gamma)$ and  any function $g\in L^2(\R,\gamma)$ admits an  orthogonal expansion  of the form
\begin{equation}\label{hexp1}
g(x)= \sum_{m=0} ^\infty c_m  H_m(x).
\end{equation}
 If $g\in L^2(\R,\gamma)$ has  the expansion  (\ref{hexp1}), we define the operator $T_1$ by
\begin{equation} \label{t1a}
T_1(g)(x) = \sum_{m=1}^\infty c_m H_{m-1}(x) \,.
\end{equation}
To simplify the notation we will write $T_1(g) =g_1$.

Suppose that $F$ is a random variable in the first Wiener chaos of $\XX$ of the form $F= I_1(\varphi)$, where $\varphi \in \HH$ has norm one. In view of the relation between Hermite polynomials and multiple stochastic integrals, it follows that for any $g\in L^2(\R, \gamma)$ of the form (\ref{hexp1}), the random variable $g(F)$ admits the Wiener chaos expansion
\begin{equation} \label{expg}
g(F)= \sum_{m=0} ^\infty c_m I_m(\varphi^{\otimes m}).
\end{equation}
Next we establish the connection between the shift operator    $T_1$ defined in  (\ref{t1a}) and the representation of a centered and square integrable random variable  as divergence given in (\ref{deltad}).

\begin{lemma} \label{lem2.1}
Let $F$ be a random variable in the first Wiener chaos of $\XX$ of the form $F= I_1(\varphi)$, where $\| \varphi \|_{\HH} =1$.  Suppose that $g\in L^2(\R, \gamma)$  is centered. Then
\[
g_1(F) \varphi = -DL^{-1} g(F).
\]
As a consequence, $g(F) = \delta(  g_1(F) \varphi)$.
\end{lemma}
\begin{proof}
Using the Wiener  chaos expansion (\ref{expg}), we can write
\[
L^{-1} g(F)=- \sum_{m=1} ^\infty  \frac { c_m} m H_m(F),
\]
which implies, taking into account that $H_m' =mH_{m-1}$, that
\[
-DL^{-1} g(F)= \sum_{m=1} ^\infty  c_mH_{m-1}(F) \varphi= g_1(F)\varphi.
\]
Property $g(F) = \delta(  g_1(F) \varphi)$ is a consequence of (\ref{deltad}).
This completes the proof.
\end{proof}

For any $k\ge 2$, we can define the iterated operator $T_k= T_1\circ \stackrel{k} { \cdots} \circ T_1$ by 
\begin{equation} \label{t1}
T_k(g)(x) = \sum_{m=k}^\infty c_m H_{m-k}(x) \,.
\end{equation}
We will write $T_k(g)= g_k$ and we have the representation 
   \begin{equation}\label{g.intrep}
   	 g(F) = \delta^k( g_k(X) \varphi^{\otimes k}) \,,
   \end{equation}
   provided  $F$ is a random variable in the first Wiener chaos of $\XX$ of the form $F= I_1(\varphi)$, with  $\| \varphi \|_{\HH} =1$, and $g$ has Hermite rank $k$.

\begin{lemma}
	Let $F$ be a random variable in the first Wiener chaos of $\XX$ of the form $F= I_1(\varphi)$, with  $\| \varphi \|_{\HH} =1$.
	Suppose that $g\in L^2(\R, \gamma)$  is centered. Then for any
 $p \geq 1$,
    \begin{equation}
    	\| g_1(F) \|_{L^p(\Omega)} \leq \sqrt{ \pi}\|g(F)\|_{L^p(\Omega)} \,.
    \end{equation}
\end{lemma}

\begin{proof}
Observe that, using Lemma  \ref{lem2.1}, we can write
     \[
     \| g_1(F)\|_{L^p(\Omega)} = \|-DL^{-1} g(F)\|_{L^p(\Omega; \mathfrak{H})} \,.
     \]
Then, using (\ref{equ5}), Minkowski's inequality and Propositions 3.2.4 and 3.2.5 of \cite{CBMS}, we can write
     \begin{eqnarray*}
		 \|-DL^{-1} g(F)\|_{L^p(\Omega; \mathfrak{H})} & \leq & \left\| \int_0^\infty D P_t g(F) dt \right\|_{L^p(\Omega; \mathfrak{H})} \\
		  & \leq & \int_0^\infty \| D P_t g(F) \|_{L^p(\Omega; \mathfrak{H})} dt \\
		  & \leq & \int_0^\infty t^{-\frac{1}{2}} e^{-t} \| g(F) \|_{L^p(\Omega)} dt \,,
	 \end{eqnarray*}
	 which allows us to complete the proof.
	 \end{proof}
	 By iteration, we  obtain
	     \begin{equation}
    	\| g_k(F) \|_{L^p(\Omega)} \leq  \pi^{\frac k2} \|g(F)\|_{L^p(\Omega)} \,,
    \end{equation}
    for any $k\ge 2$, provided $g$ has Hermite rank $k$ and  $F= I_1(\varphi)$, with  $\| \varphi \|_{\HH} =1$. If $g$   has Hermite rank strictly less than $k$, we can write
    \[
T_kg(x)= T_k \widetilde{g}(x),
\]
where $\widetilde{g}(x)= \sum_{m=k} ^\infty c_m H_m(x)$. Then,
\[
\| T_k g(F)\|_{ L^p(\Omega)} \le  \pi^{\frac k2} \| g(F) \|_{ L^p(\Omega)}  +  \pi^{\frac k2} \left \| \sum_{m=0 }^{k-1} c_m H_m(F)  \right\|_{ L^p(\Omega)}   \le  \pi^{\frac k2} \|g(F) \|_{ L^p(\Omega) } + C_{k,p}.
\]
 
Consider $\mathfrak{H}=\R$,  the probability space $(\Omega,  \mathcal{F}, P)= (\R, \mathcal{B}(\R), \gamma)$  and the isonornal Gaussian process $\XX(h)=h$.   For any $k\ge0$ and $p\ge 1$, denote by $\mathbb{D}^{k,p}(\R,\gamma)$ the  corresponding Sobolev spaces of functions.    Notice that if $g\in \mathbb{D}^{k,p}(\R,\gamma)$,  and $F=I_1(\varphi)$ is an element in the first Wiener chaos of a general isonormal Gaussian process $\XX$, then $g(F)\in \mathbb{D}^{k,p}$.

% In  this context,   
 %for any $k\ge 1$, we denote by $\mathcal{D}^k$  the  class of functions
 %\[
%\mathcal{D}^k := \cap_{p\ge 1}\mathbb{D}^{k,p}(\R).
%\]
%That is, a function $g\in \mathcal{D}^k$    has $k$ weak derivatives  in  $L^p(\R,\gamma)$ for all $p\ge 1$.
 
 The next lemma  provides a regularizing property of the operator $T_k$.
 \begin{lemma}  \label{lem2.4}
 Suppose that $g\in \mathbb{D}^{j,p}(\R,\gamma)$ for some $j\ge 0$ and $p>1$.  Then  $T_kg \in \mathbb{D}^{j+k}(\R,\gamma)$ for all $k\ge 1$.
 \end{lemma}
 
\begin{proof}
We can assume that $g$ has Hermite rank $k$, otherwise, we just subtract the first $k$ terms in its expansion. 
Then, the result is an immediate consequence of the fact that $T_k=(-DL^{-1})^k$ and  the equivalence in $L^p(\R,\gamma)$ of the operators $D$ and $(-L)^{1/2}$, which follows from Meyer's inequalities (see, for instance, \cite{nualartbook}). 
\end{proof}

Notice that  $T_1$ and the derivative operator do not commute. We will write $(g_1)' = g'_1$, which is different from
$T_1(g')$. Indeed,  for any $g\in L^2(\R, \gamma)$, we have
\[
g'_1 = T_1(g') - g_2,
\]
because if $g$ has the expansion (\ref{hexp1}), we obtain
\[
g'_1(x) = \sum_{m=2} ^\infty c_m (m-1) H_{m-2} (x),
\]
\[
 T_1(g')(x) = \sum_{m=2} ^\infty c_m m H_{m-2} (x)
\]
and
\[
g_2(x) = \sum_{m=2} ^\infty c_m H_{m-2} (x).
\]
More generally we can show that for any $k,l \ge 1$,
	\[
	 g_k^{(l)} = \sum_{i=0}^l \binom{l}{i}\alpha_{k,i} T_{k+i}(g^{(l-i)}),
	\]
where $\alpha_{k,i}= (-1)^i k(k+1)\cdots(k+i-1)$, with  the convention $\alpha_{k,i}=1$ if $i=0$.

\subsection{Brascamp-Lieb inequality}

In this subsection we recall  a version of the rank-one Brascamp-Lieb inequality  that will be  intensively used  through   this paper (see  \cite{barthe,bcct,bl} and the references therein). This inequality constitutes a generalization of both H\"older's and Young's convolution inequalities. 

\begin{prop} \label{prop2.10} Let $2\le M\le N$ be fixed integers. 
	Consider  nonnegative measurable functions $f_j: \R \rightarrow \R_+$, $1\le j\le N$, and fix nonzero vectors ${\bf v_j} \in \R^M$.
	Fix  positive numbers $p_j$,  $1\le j\le N$, verifying the following conditions:
	\begin{itemize}
	\item[(i)]  $\sum_{j=1}^N p_j = M$,
	\item[(ii)]  For any subset $I\subset \{1,\dots, M\}$, we have $\sum_{j\in I} p_j \le {\rm dim} \left( {\rm Span} \{ {\bf v}_j, j\in I \} \right)$.
	\end{itemize}
	Then, there exists a  finite constant $C$, depending on $N, M$ and the  $p_j$'s such that
	\begin{equation}  \label{BL}
	\sum_{ {\bf k} \in \mathbb{Z}^M} \prod_{j=1} ^N  f_j({\bf k} \cdot {\bf v} _j)  \le C  \prod_{j=1} ^N  \left( \sum_{ k \in \mathbb{Z}}
	f_j (k)^{ 1/p_j} \right) ^{p_j}.
	\end{equation}
\end{prop}

\subsection{Stein's method}
Let  $h: \mathbb{R} \to \mathbb{R}$ be a Borel function such that $h \in L^1(\R, \gamma)$. The ordinary differential equation
      \begin{equation} \label{stein}
      f'(x) - xf(x) = h(x) - \mathbb{E}(h(Z))
      \end{equation}
is called Stein's equation associated with $h$. The function 
\[
f_h(x):= e^{x^2/2}\int_{-\infty}^x (h(y) - \mathbb{E}(h(Z)))e^{-y^2/2} dy
\]
 is the unique solution to the Stein's equation satisfying $\lim_{|x| \to \infty} e^{-x^2/2} f_h(x) = 0$. Moreover, if $h$ is bounded, $f_h$ satisfies
      \begin{equation}  \label{equ81}
      \|f_h \|_\infty \leq \sqrt{\frac{\pi}{2}} \|h - \mathbb{E}(h(Z)) \|_\infty 
      \end{equation}
      and
      \begin{equation}  \label{equ82}
       \|f_h' \|_\infty \leq 2\|h - \mathbb{E}(h(Z)) \|_\infty
       \end{equation}
(see \cite{np-book} and the references therein). 

We recall that the total variation distance between the laws of two random variables $F,G$ is defined by
\[
d_{\rm TV}(F,G) = \sup_{B \in \mathcal{B}(\mathbb{R})}|P(F \in B) - P(G \in B)| \,,
\]
 where the supremum runs over all Borel sets $B \subset \mathbb{R}$.  Substituting $x$ by $F$  in  Stein's equation (\ref{stein}) and
 using the  inequalities (\ref{equ81}) and (\ref{equ82}) lead  to the fundamental estimate
\begin{equation}  \label{equ83}
d_{\rm TV}(F,Z) =  \sup_{f\in  \mathcal{C}^1(\R),  \| f\|_\infty \le \sqrt{ \pi/2}, \| f' \|_\infty \le 2 } | \EE  (f'(F)- Ff(F)) | \,.
\end{equation}

\section{Basic estimates for the total variation distance}
In the framework of an isonormal Gaussian process $\XX$, we  can use Stein's equation to estimate  the  total variation distance between a random variable $F = \delta(u)$ and $Z$.  
 First let us recall the following basic result (see \cite{np-book}), which is an easy consequence of  (\ref{equ83}) and the duality relationship (\ref{dua}).
 
\begin{prop}\label{stein.hd0}
	Assume that $u\in {\rm Dom} \,\delta$,  $F=\delta(u) \in \mathbb{D}^{1,2}$ and $\EE(F^2)=1$.  Then,
	\begin{eqnarray*}
	d_{\rm TV} (F,Z) 	\le   2 \EE( |1-\langle DF, u \rangle_{\mathfrak{H}}|) \,.
	\end{eqnarray*}
\end{prop}
 Notice that, applying the duality relationship (\ref{dua}), we can write
 \[
 \EE(\langle DF, u \rangle_{\mathfrak{H}})= \EE( F \delta (u))= \EE(F^2) = 1.
 \]
 As a consequence, if $F \in \mathbb{D}^{2,2}$, we apply Cauchy-Schwarz and Poincar\'e inequalities to derive the following estimate
\begin{equation}
		d_{\rm TV} (F,Z)   \leq 2 \sqrt{\EE(1-\langle DF, u \rangle_{\mathfrak{H}})^2} = 2 \sqrt{{\rm Var}(D_uF)} \leq 2 \sqrt{\EE( \|D(D_uF)\|^2_{\mathfrak{H}})} \label{2bound} \,,
	\end{equation}
	where we have used the notation $D_uF=\langle u, DF \rangle_{\mathfrak{H}}$. We will also write $ D_u^{i+1} F = \langle u, D(D_u^i F) \rangle_{\mathfrak{H}}  $ for  $i \ge 1$.
	
Furthermore, if the random variable $F$ admits higher  order derivatives,  iterating the integration by parts argument   we can improve the bound   (\ref{2bound}) as follows.  
\begin{prop}\label{stein.hd}
	Assume that  $u\in {\rm Dom} \,\delta$,  $F=\delta(u) \in \mathbb{D}^{3,2}$ and $\EE(F^2)=1$.   Then
	\begin{eqnarray*}
	d_{\rm TV} (F,Z)  	  & \leq & (8+\sqrt{32\pi}) \, \EE(\|D(D_uF)\|^2_{\mathfrak{H}}) + \sqrt{2\pi} \, |\EE (F^3)| \\
		 & & + \ \sqrt{32\pi}\,\EE(|D_u^2F|^2 )+ 4\pi \,\EE(|D_u^3 F|) \,.
	\end{eqnarray*}
\end{prop}
\begin{proof}
Fix a  continuous function $h: \R \rightarrow [0,1]$. 
	Using Stein's equation (\ref{stein}), there exists a    function $f_h\in \mathcal{C}^1(\R)$ such that $\|f_h\|_\infty \leq \sqrt{\frac{\pi}{2}}$ and $\|f_h'\|_\infty \leq 2$, satisfying
	\[
	I:=|\EE(h(F)) - \EE(h(Z))| = |\EE(f_h'(F) - Ff_h(F))| \,.
	\]
	Applying  the duality relationship (\ref{dua}),  yields $$I=|\EE(f_h'(F)(1-\langle DF, u \rangle_{\mathfrak{H}}))| \,.$$
	Taking into account  that $\EE(\langle DF, u \rangle_{\mathfrak{H}} )= \EE(F^2) = 1$, we have
	    \[
	    I=|\EE\left((f_h'(F) - \EE(f_h'(Z)))(1-\langle DF, u \rangle_{\mathfrak{H}})\right)| \,.
	    \]
	Let $f_\varphi$ be the solution to Stein's equation associated with the function $\varphi= f'_h$. Then, we have
\[
I= \left|\EE\left((f_{\varphi}'(F) - Ff_{\varphi}(F))(1-\langle DF, u \rangle_{\mathfrak{H}})\right)\right|
\]
	 where $\|f_{\varphi}\|_\infty \leq 4\sqrt{\pi/2}$ and $\|f_{\varphi}'\|_\infty \leq 8$.  Substituting $F$ by $\delta (u)$ and  applying again
	 the  duality relationship (\ref{dua}), yields
	 \begin{eqnarray}
	 	I & = &\left|\EE\left(f_{\varphi}'(F)(1-D_uF) - \langle u, D(f_{\varphi}(F)(1-D_uF))\rangle_{\mathfrak{H}}\right)\right| \nonumber\\
		& = & \left|\EE\left(f_{\varphi}'(F)(1-D_uF)^2\right) + \EE\left(f_{\varphi}(F) D_u^2F\right)\right| \label{mid.ineq}\\
		& \leq & 8 \EE((1-D_uF)^2) + |\EE\left((f_{\varphi}(F) - \EE( f_{\varphi}(Z)))D_u^2 F \right)| + |\EE (f_{\varphi}(Z) )\EE(D_u^2 F)| \nonumber \\
		& =: &I_1 + I_2 + I_3 \,. \nonumber
	 \end{eqnarray}
	 For the term $I_1$, we apply Poincar\'e inequality to get $$I_1 \leq 8 \EE(\|D(D_uF)\|^2_{\mathfrak{H}})\,.$$
	 For the term $I_3$, taking into account that
	   $$\EE (D_u^2 F) = \EE(\langle u, DF \rangle_{\mathfrak{H}} \delta(u)) = \frac{1}{2} \EE (\langle u, DF^2 \rangle_{\mathfrak{H}} )= \frac{1}{2} \EE(F^3 ),$$
	  we obtain $$I_3 \leq 2 \sqrt{\pi/2} |\EE( F^3)| \,.$$
For the term $I_2$, applying Stein's equation associated with $\psi =f_{\varphi}$ yields
	   \begin{eqnarray*}
	   	  I_2 &=& \left|\EE\left((f_{\psi}'(F) - Ff_{\psi}(F)) D_u^2F\right)\right| \\
		  & \leq & \left|\EE \left(f_{\psi}'(F)(D^2_uF - D_uF D_u^2F)\right)\right| + \left|\EE \left(f_{\psi}(F) D_u^3 F\right)\right|,
	   \end{eqnarray*}
	 where $f_{\psi}$ satisfies $\|f_{\psi}\|_\infty \leq  4\pi  $ and $\|f_{\psi}'\|_\infty \leq   16\sqrt{\pi/2}$. Finally,
	   $$\EE(|D^2_uF - D_uF D_u^2F| )\leq \frac{1}{2}\left(\EE(|D_u^2 F|^2) + \EE(|1-D_uF|^2)\right) \leq \frac{1}{2}\left(\EE(|D_u^2 F|^2) + \EE(\|DD_uF\|^2_{\mathfrak{H}})\right) \,. $$
	   This concludes the proof of the proposition.
\end{proof}

If we bound \eqref{mid.ineq} in a different way, we would get the following estimate.
\begin{prop}\label{stein.hd2}
	Assume  that  $u\in {\rm Dom} \,\delta$, $F=\delta(u) \in \mathbb{D}^{2,2}$ and $\EE(F^2)=1$.  Then 
	\[
 	d_{\rm TV} (F,Z)   \leq 8 \EE((1-D_uF)^2 )+ \sqrt{8\pi} \EE(|D_u^2 F|)\,.
	\]
\end{prop}

%============================================
\section{Main results}\label{main.results}
%============================================
Consider a centered stationary Gaussian  family of random variables   $X=\{ X_n , n\in \Z\}$  with unit variance and covariance
    $\rho(k) = \EE(X_0 X_k)$ for $k \in \mathbb{Z}$. 
    Define the Hilbert space $\mathfrak{H}$ as the closure of the linear span of $\mathbb{Z}$ under the inner product
    $    \langle j,k \rangle_{\mathfrak{H}} = \rho(j-k)$. 
    The mapping $k \rightarrow X_k$ can be extended to a linear isometry from $\mathfrak{H}$ to  the  closed linear subspace $L^2(\Omega)$ spanned by $X$. Then $\{X_{\varphi}, \varphi \in \mathfrak{H}\}$ is  an isonormal Gaussian process.
    
    Consider the sequence   $	  Y_n := \frac{1}{\sqrt{n}} \sum_{j=1}^n g(X_j) $ introduced in (\ref{yn}), where $g\in L^2(\R, \gamma)$ has  Hermite rank $d\ge 1$ and let 
	$\sigma_n^2= \EE (Y_n^2)$.   	Under condition  (\ref{h1}), it is well known that as $n\to \infty$, $\sigma_n^2 \to \sigma^2$,  where $\sigma^2$ has been defined in (\ref{bm.sig}).
		 
Along the paper, we will denote by $C$ a generic constant, whose value can  be different from one formula to another one. 	
	 
Our aim is to establish estimates on the total variation distance between $Y_n/\sigma_n$ and $Z$.  We will make use of  the representation $Y_n = \delta (u_n)$, where 
\begin{equation} \label{un}
u_n  =   \frac{1}{\sqrt{n}} \sum_{j=1}^n g_1(X_j) j,
\end{equation}
given by Lemma \ref{lem2.1}.
  Then,  if $g\in \D^{2,2} (\R,\gamma)$, by inequality \eqref{2bound} and taking into account that
$\sigma_ n \rightarrow \sigma>0$, we have the estimate
\begin{eqnarray}
	d_{TV} (Y_n / \sigma_n ,Z) &\leq& \frac{1}{\sigma_n^2}\sqrt{\EE  ( |  \langle DY_n,  u_n \rangle_{\HH} -\sigma_n^2| ^2 ) }  \nonumber \\
	&\leq & C\sqrt{\EE \left( \left\| D (\langle DY_n, u_n \rangle_{\HH})\right \| ^2_{\HH} \right)}= C\sqrt{A_1}, \label{yn2.est}
\end{eqnarray}
where $A_1=\EE (\|DD_{u_n} Y_n\|^2_{\mathfrak{H}})$.
Furthermore, using Proposition \ref{stein.hd2}, we can write
\begin{eqnarray}
	d_{TV} (Y_n / \sigma_n ,Z) & \le & \frac{8}{\sigma_n^4} \EE(\|DD_{u_n} Y_n\|^2_{\mathfrak{H}}) + \frac{\sqrt{8\pi}}{\sigma_n^3} \sqrt{\EE(|D_{u_n}^2 Y_n|^2)} \nonumber \\
	   & & \le  C(A_1 + \sqrt{A_2}) \,, \label{dist.ynz2}
\end{eqnarray}
where $A_2=\EE(|D_{u_n}^2 Y_n|^2)$ and
where we recall that   $D_{u_n} Y_n = \langle u_n, DY_n \rangle_{\mathfrak{H}}$ and $D^i_{u_n}Y_n = \langle u_n, D^{i-1}_{u_n}Y_n \rangle_{\mathfrak{H}}$ for $i\ge 2$. 

 If  $g\in \D^{3,2} (\R,\gamma)$, using Proposition \ref{stein.hd}, we obtain
\begin{eqnarray}
d_{TV} (Y_n / \sigma_n ,Z) & \le & \frac{8+\sqrt{32\pi}}{\sigma_n^4} \EE(\|DD_{u_n} Y_n\|^2_{\mathfrak{H}}) + \ \frac{\sqrt{32\pi}}{\sigma_n^6} \EE(|D^2_{u_n}Y_n|^2 ) \nonumber \\
   & & + \frac{\sqrt{2\pi}}{\sigma_n^3} |\EE(Y_n^3)|  + \frac{4\pi}{\sigma_n^4} \sqrt{\EE(|D^3_{u_n} Y_n|^2)} \nonumber \\
   & &\le C(A_1 + A_2 + A_3 + A_4) \,,\label{dist.ynz}
\end{eqnarray}
where  $A_3=|\EE(Y_n^3)|$ and $A_4=\sqrt{\EE(|D^3_{u_n} Y_n|^2)}$.  

In the sequel we will derive estimates on the terms $A_i$, $i=1, \dots, 4$ in terms of the covariance function $\rho(k)$.
We use the notation $A_i \prec A_j$ if $A_i$'s bound has a better convergence rate to zero than that of $A_j$.  To get the best possible rate, we use the following strategy. If $g$ is just twice differentiable, we can use
the estimates (\ref{yn2.est}) and (\ref{dist.ynz2}). Then we will 
 compare the rates of the terms $A_1$ and $A_2$. If $A_1 \prec A_2$, we just use the bound \eqref{yn2.est}. Otherwise, \eqref{dist.ynz2} would be used. If $g$ has higher order derivatives, 
we would use the bound \eqref{dist.ynz} if  $A_2 \prec \sqrt{A_1}$ and the rates of $A_3$ and $A_4$ are better than  those of   $\sqrt{A_2}$ and $  \sqrt{A_1}$. Otherwise, if the rate of either $A_3$ or $A_4$ is worse than  that of $\sqrt{A_1}$ or $\sqrt{A_2}$, we consider the bound \eqref{dist.ynz2} or \eqref{yn2.est} depending on the comparison between $A_2$ and $A_1$.  
 
 Before presenting the main results, we will derive some expressions and estimates for the terms $A_i$, $i=1,2,4$. 
To simplify the notation, we will write   $\rho_{ij} = \rho(l_i-l_j)$ for any $1 \le i,j \le n$.

\begin{lemma}\label{var.yn} Suppose that $g\in \D^{2,4} (\R,\gamma)$. Then, 
	\[
	A_1 \leq \frac{2}{n^2} \sum_{i=1}^2 \sum_{l_1,l_2,l_3,l_4=1}^n \left| \EE(I_i) \rho(l_1-l_2) \rho(l_3-l_4) \rho(l_2-l_4) \right|,
	\]
where
   \begin{equation}
    I_1 = g''(X_{l_2}) g''(X_{l_4})   g_1(X_{l_1}) g_1(X_{l_3}) \,, \label{I1}
    \end{equation}
    and
    \begin{equation}
    I_2 = g'(X_{l_1}) g'(X_{l_3})  g'_1(X_{l_2}) g'_1(X_{l_4}) \label{I2} \,.
   \end{equation}
\end{lemma}

\begin{proof}
	First, we have
	\begin{eqnarray}
	 \EE \left( \left\| D (\langle DY_n, u_n \rangle_{\HH})\right \| ^2_{\HH} \right) \notag & \le& 2 \EE \left( \left\|  D^2Y_n\otimes _1 u_n \right \| ^2_{\HH} \right)
	 +2\EE \left( \left\|  \langle D_*Y_n, Du_n(*) \rangle_{\HH}\right \| ^2_{\HH} \right),  \label{A1}
	 \end{eqnarray}
	 where $D^2Y_n\otimes _1 u_n$ denotes the contraction of one variable between $D^2Y_n$ and $u_n$ 
	 and 
	 \[
	  \langle D_*Y_n, Du_n(*) \rangle_{\HH} =\sum_{i=1} ^\infty \langle DY_n ,e_i \rangle_{\HH}  D( \langle u_n, e_i \rangle_{\HH}),
	  \]
	with $\{e_i, i\ge 1\}$ being a complete orthonormal system in $\HH$.  
 This implies, taking into account (\ref{un}), that
	\[
	 D^2Y_n \otimes _1 u_n  =\frac 1n  \sum_{j,k=1}^n  g''(X_k)  g_1(X_j)  \rho(j-k)   k
	\]
	and
	\[
	  \langle D_*Y_n, Du_n(*) \rangle_{\HH}
	   =\frac 1n  \sum_{j,k=1}^n g'(X_j)  g_1'(X_k)\rho(j-k) k \,.
	\]
	As a consequence,
	\[
	 \left\|  D^2Y_n \otimes _1 u_n\right \| ^2_{\HH} = \frac 1{n^2}  \sum_{l_1,l_2,l_3,l_4=1}^n  I_1 \rho(l_1-l_2) \rho(l_3-l_4) \rho(l_2-l_4),
	 \]
		and
	\[
	 \left\|\langle D_*Y_n, Du_n(*) \rangle_{\HH}\right \| ^2_{\HH} = \frac 1{n^2}  \sum_{l_1,l_2,l_3,l_4=1}^n  I_2 \rho(l_1-l_2) \rho(l_3-l_4) \rho(l_2-l_4),
	 \]
	 which implies the desired result.
\end{proof}

Next we derive a simple estimate for  the term $A_2$, assuming again that $g\in \D^{2,6} (\R,\gamma)$. Notice that
  \[
  D_{u_n}Y_n = \frac{1}{n} \sum_{l_1, l_2=1}^n g_1(X_{l_1})g'(X_{l_2}) \rho(l_1 - l_2) \,.
  \]
Denote 
\begin{equation}  \label{f1}
f_{1}(l_1, l_2, l_3) =  g_1'(X_{l_1})g'(X_{l_2}) g_1(X_{l_3}) \,
\end{equation}
and
       \begin{equation}  \label{f2}
       f_{2}(l_1, l_2, l_3)=g_1(X_{l_1})g''(X_{l_2}) g_1(X_{l_3}) \,.
       \end{equation}
Correspondingly, using the notation $\rho_{ij} = \rho(l_i-l_j)$, we can write
\[
  	 D^2_{u_n} Y_n = \frac{1}{\sqrt{n^3}} \sum_{l_1, l_2, l_3=1}^n   \Big (f_{1}(l_1, l_2, l_3) \rho_{12} \rho_{13}  + \ f_{2}(l_1, l_2, l_3) \rho_{12} \rho_{23}\Big) \,.
\]
Thus, 
  \begin{eqnarray}
  	  A_2 = \EE((D^2_{u_n}Y_n)^2) &\leq&  \frac{2}{n^3} \sum_{l_1, \ldots, l_6=1}^n \Big( \EE (f_{1}(l_1, l_2, l_3) f_{1}(l_4, l_5, l_6)) \rho_{12} \rho_{13} \rho_{45} \rho_{46} \nonumber\\
	 & & \qquad \quad + \ \EE (f_{2}(l_1, l_2, l_3) f_{2}(l_4, l_5, l_6)) \rho_{12} \rho_{23} \rho_{45} \rho_{56} \Big) \,.\label{est.a2}
  \end{eqnarray}
    Finally, let us  compute the term $A_4$, assuming $g\in \D^{3,8} (\R,\gamma)$. We have
    \begin{eqnarray*}
    D^3_{u_n} Y_n &= &  \frac{1}{n^2} \sum_{l_1, l_2, l_3,l_4=1}^n  \sum_{i=1}^3  \Big (f_{1}^{(i)}(l_1, l_2, l_3) g_1(X_{l_4}) \rho_{12} \rho_{13}\rho_{i4} \\
    && + f_{2}^{(i)}(l_1, l_2, l_3) g_1(X_{l_4}) \rho_{12} \rho_{23}\rho_{i4}\Big) \,,
 \end{eqnarray*}
 where 
    \begin{eqnarray*}
f_{1}^{(1)}(l_1, l_2, l_3) &= & g_1''(X_{l_1})g'(X_{l_2}) g_1(X_{l_3}) \,, \\
f_{1}^{(2)}(l_1, l_2, l_3) &= & g_1'(X_{l_1})g''(X_{l_2}) g_1(X_{l_3}) \,, \\
f_{1}^{(3)}(l_1, l_2, l_3) &= & g_1'(X_{l_1})g'(X_{l_2}) g'_1(X_{l_3}) 
\end{eqnarray*}
and
   \begin{eqnarray*}
f_{2}^{(1)}(l_1, l_2, l_3) &= & g_1'(X_{l_1})g''(X_{l_2}) g_1(X_{l_3}) \,, \\
f_{2}^{(2)}(l_1, l_2, l_3) &= & g_1(X_{l_1})g'''(X_{l_2}) g_1(X_{l_3}) \,, \\
f_{2}^{(3)}(l_1, l_2, l_3) &= & g_1(X_{l_1})g''(X_{l_2}) g'_1(X_{l_3}).
\end{eqnarray*}
Therefore,
\begin{eqnarray}  \nonumber
A_4^2&=&\EE((D^3_{u_n}Y_n)^2)\\  \nonumber
 & \leq& \frac{2}{n^4}  \sum_{i=1}^3  \sum_{ j=1, \ldots, 8} \sum_{l_j=1}^n 
 \EE \left(f_{1}^{(i)}(l_1, l_2, l_3)  g_1(X_{l_4}) f_{1}^{(i+4)}(l_5, l_6, l_7)  g_1(X_{l_8})\right)    \\  \nonumber
&&   \times \rho_{12} \rho_{13}\rho_{i4}  \rho_{56} \rho_{57}\rho_{(i+4) 8}  \\  \nonumber
&&+ \frac{2}{n^4}   \sum_{i=1}^3  \sum_{ j=1, \ldots, 8} \sum_{l_j=1}^n
 \EE \left(f_{2}^{(i)}(l_1, l_2, l_3)  g_1(X_{l_4}) f_{2}^{(i+4)}(l_5, l_6, l_7)  g_1(X_{l_8})\right)    \\ \label{equ90}
&&   \times \rho_{12} \rho_{23}\rho_{i4}  \rho_{56} \rho_{67}\rho_{(i+4) 8}.
\end{eqnarray}

We are now ready to state and prove the main results of this paper. The notation is that of Theorem \ref{bm}.

\subsection{Case  $d=1$}

\begin{thm}\label{main.d1}
Let $d=1$ and $g \in \D^{2,4} (\R,\gamma)$. 
  Suppose that (\ref{h1}) holds true.  Then    
      \[
      d_{\rm TV}(Y_n / \sigma_n, Z) \leq C n^{-\frac{1}{2}} \,.
      \]
\end{thm}
\begin{proof}
We use  the inequality   \eqref{yn2.est}   and we need to estimate the term $A_1$.  By Lemma \ref{lem2.4}, H\"older's inequality and 
the fact that $g \in \D^{2,4} (\R,\gamma)$,
 the  quantities $I_1 $ and $ I_2$  have finite expectation. Then
    \[
     A_1 \leq \frac{C}{n^2} \sum_{l_1, l_2, l_3, l_4 = 1}^n |\rho(l_1 - l_2) \rho(l_3 - l_4) \rho(l_2 - l_4)| \,.
     \]
Making the change of variables $k_1= l_1 -l_2$, $k_2= l_3-l_4$, $ k_3= l_2-l_4$ and using condition (\ref{h1}) with $d=1$,   we  obtain
	   \[
	   A_1 \leq \frac{C}{n} \sum_{|k_i| \leq n, 1\le i\le 3} |\rho(k_1)\rho(k_2)\rho(k_3)| \leq \frac{C}{n} \,,
	   \]
which provides the desired estimate.
\end{proof}

\subsection{Case of $d=2$}
\begin{thm}\label{main.d12}
	Let $d=2$ and suppose that (\ref{h1}) holds true.  
	\begin{itemize}
\item[(i)] If $g \in  \D^{2,4} (\R,\gamma)$, we have
		        \[
		        d_{\rm TV}(Y_n / \sigma_n, Z) \leq C n^{-\frac{1}{2}} \left(\sum_{|k| \leq n} |\rho(k)|\right)^{\frac{3}{2}} \,.
		        \]
	\item[(ii)]	    If $g \in \D^{3,4} (\R,\gamma)$, we have
		    \[
		    d_{\rm TV}(Y_n / \sigma_n, Z) \leq C n^{-\frac{1}{2}} \sum_{|k| \leq n} |\rho(k)| \,.
		    \]
		\item[(iii)]	If $g \in \D^{4,4} (\R,\gamma)$, we have
		\[
				   d_{\rm TV}(Y_n / \sigma_n, Z)  \leq  C n^{-\frac{1}{2}} \left(\sum_{|k| \leq n} |\rho(k)|\right)^{\frac{1}{2}}  + C n^{-\frac{1}{2}} \left(\sum_{|k| \leq n} |\rho(k)|^{\frac{4}{3}}\right)^{\frac 32} \,.
		\]
		\item[(iv)]	If $g \in  \D^{5,6} (\R,\gamma)$, we have	
		\[
				   d_{\rm TV}(Y_n / \sigma_n, Z)  \leq  C n^{-\frac{1}{2}} \left(\sum_{|k| \leq n} |\rho(k)|\right)^{\frac{1}{2}}  + C n^{-\frac{1}{2}} \left(\sum_{|k| \leq n} |\rho(k)|^{\frac{3}{2}}\right)^{2} \,.
		\]
		  \item[(v)] 
			If $g \in \D^{6,8} (\R,\gamma)$, we have
			   \[
			   d_{\rm TV}(Y_n / \sigma_n, Z) \leq    
			   C n^{-\frac{1}{2}} \left(\sum_{|k| \leq n} |\rho(k)|^{\frac{3}{2}}\right)^{2} .
			   \]
			   \end{itemize}
\end{thm}

\begin{remark}
For $g\in  \D^{6,8} (\R,\gamma)$ the rate estalbished  in point (v) coincides with the rate for  the Hermite polynomial $g(x)=x^2-1$, obtained by   Bierm\'e, Bonami, Nourdin and Peccati in
\cite{bbnp} using the  optimal bound for the total variation distance in the case of random variables in a fixed Wiener chaos derived 
by Nourdin and Peccati in \cite{np-15} (see Proposition \ref{bd.4m-2}).  When the function $g $ belongs to  $ \D^{i,4}(\R,\gamma)$, $2\le i \le 4$ or $  g\in \D^{5,6}(\R,\gamma)$, the rates we have obtained are worse than the rate for $g\in  \D^{6,8} (\R,\gamma)$. For  $  g\in \D^{i,4}(\R,\gamma)$, $i=2,3,4$, the estimates in points (i), (ii) and (iii) will be established using Proposition \ref{stein.hd0}, whereas, for $  g\in \D^{5,6}(\R,\gamma)$   we will use  Proposition \ref{stein.hd2}  to derive the estimate in point (iv) and for   $g\in  \D^{6,8} (\R,\gamma)$ we apply Proposition \ref{stein.hd}.
\end{remark}

\begin{proof}[Proof of Theorem \ref{main.d12}]
The proof will be done in several steps.  

\medskip
\noindent {\it Case $g \in \D^{2,4} (\R,\gamma)$}.  \quad 
We apply Lemma \ref{var.yn} to derive the rate of convergence  of $A_1$.   Using  arguments  similar to those in the case $d=1$ yields
       \begin{equation}  \label{equ2}
       A_1 \leq \frac{C}{n} \sum_{|k_i| \leq n, 1\le i\le 3} |\rho(k_1)\rho(k_2)\rho(k_3) |= \frac{C}{n} \left(\sum_{|k| \leq n} |\rho(k)| \right)^3\,,
       \end{equation}
       which gives the desired estimate in view of   \eqref{yn2.est}.
    
    We claim that, even if we impose more integrability  conditions on the function $g$, that is, $g \in \D^{2,6} (\R,\gamma)$,  the estimate  (\ref{dist.ynz2}) does not  give a rate better than  (\ref{equ2}). In fact, let  us estimate the term $A_2$, which is bounded by the inequality  \eqref{est.a2}, where $f_1$ and $f_2$ are defined in (\ref{f1}) and (\ref{f2}).   The term $\EE (f_{2}(l_1, l_2, l_3) f_{2}(l_4, l_5, l_6))$ cannot be integrated by parts because it involves $g''$ and  $g$ is only twice weakly differentiable. Therefore, if  $g\in  \D^{2,6} (\R,\gamma)$,  using   Lemma  \ref{lem2.4} together with H\"older's inequality,  and making a change of variables, we obtain
  \begin{eqnarray*}
  A_2 & \le & \frac C{n^3}   \sum_{l_1, \ldots, l_6=1}^n  \Big( | \rho_{12}  \rho_{13} \rho_{45}\rho_{46} |+ |\rho_{12}\rho_{23}\rho_{45}\rho_{56}|  \Big)  \\
  &\le & \frac Cn   \sum_{ |k_i | \le n , 1\le i\le 4}  \prod_{i=1}^4 |\rho(k_i) |   = \frac Cn  \left(\sum_{|k| \le n}  | \rho(k)|\right)^4.
  \end{eqnarray*}
Thus,   $A_1 \prec A_2$, so  we use  (\ref{yn2.est})  and  (\ref{equ2}) gives the best rate.

\medskip	
\noindent {\it  Case $g \in \D^{3,4}(\R,\gamma)$}. \quad Let us first estimate the term $A_1$. 
Because $g$ has three derivatives, using    Lemma \ref{var.yn} and Lemma \ref{cov.pth1},  we obtain
      \[
       A_1 \leq \frac{C}{n^2} \sum_{l_1, l_2, l_3, l_4 = 1}^n |\rho(l_1 - l_2) \rho(l_3 - l_4) \rho(l_2 - l_4)| \sum_{j \neq 1} 
       |\rho(l_1 - l_j)| \,.
       \]
 Making the change of variables $l_1-l_2=k_1$, $l_2-l_4=k_2$ and $l_3-l_4=k_3$, yields
      \begin{eqnarray*}
      	A_1 & \leq & \frac{C}{n} \sum_{|k_i| \leq n}  \Big( |\rho^2(k_1) \rho(k_2) \rho(k_3)| + |\rho(k_1) \rho(k_2) \rho(k_3) \rho(k_1 + k_2)| \\
 		& & \ \ + \ |\rho(k_1) \rho(k_2) \rho(k_3) \rho(k_1 + k_2 - k_3)| \Big).
      \end{eqnarray*}
      Taking into account condition (\ref{h1}) and   applying    (\ref{equ21})  with  $M=3$, yields
         \begin{equation} \label{equ1}
        	A_1  \leq  \frac Cn \left(\sum_{|k| \leq n} |\rho(k)|\right)^2,
	\end{equation}
	which gives the desired estimate in view of   \eqref{yn2.est}.
        
        Again, we claim that imposing more integrability conditions and using  either  (\ref{dist.ynz2}) or the more refined estimate
         (\ref{dist.ynz}) does not improve the above rate. Indeed, let us first estimate the term $A_2$, assuming $g\in \D^{3,6}(\R,\gamma)$.
         Because  $g$ is three times weakly differentiable,  we can integrate by parts once in the expectations appearing in
         \eqref{est.a2}.
         The two summands   in \eqref{est.a2} are similar, thus it suffices to consider the first one. Recall  that
$f_{1}(l_1, l_2, l_3) =  g_1'(X_{l_1})g'(X_{l_2}) g_1(X_{l_3}) $ has been defined in (\ref{f1}). 
Using the representation $g'(X_{l_2})= \delta(T_1(g')(X_{l_2}) l_2)$, applying  the duality relationship (\ref{dua}),  and making a change of variables, we obtain
\begin{eqnarray*}
	A_2 &\leq& \frac{C}{n^3} \sum_{l_1, \ldots, l_6=1}^n  \Big( \rho_{12}^2  |\rho_{13} \rho_{45}\rho_{46} |+ |\rho_{12}\rho_{13}\rho_{45}\rho_{46}\rho_{23}| +| \rho_{12}\rho_{13}\rho_{45}\rho_{46}| \sum_{i=4}^6 |\rho_{2i}| \Big) \\
	  &\leq& \frac{C}{n^2} \sum_{ |k_i | \le n \atop 1\le i\le 5}  \Big(\rho(k_1)^2 \prod_{i=2}^4 |\rho(k_i)| + |\rho(k_1 - k_2)|\prod_{i=1}^4 |\rho(k_i) | + \prod_{i=1}^5| \rho(k_i)| \Big) \,.
\end{eqnarray*}
This implies, using (\ref{equ21}) with  $M=4$ for the second summand, that
\[
A_2 \le  \frac Cn\left(\sum_{|k|\leq n} |\rho(k)|\right)^3  + \frac C {n^2}\left(\sum_{|k|\leq n} |\rho(k)|\right)^5 
\le \frac  Cn\left(\sum_{|k|\leq n} |\rho(k)|\right)^3,
\]
where we have used the fact that  $\sum_{|k|\leq n} |\rho(k)| \le C \sqrt{n}$ in the second inequality. 
 Clearly,  $A_1 \prec A_2$. So the estimate (\ref{yn2.est})   is better than  (\ref{dist.ynz2}).
        
On the other hand, the estimate  (\ref{dist.ynz})  does not provide a rate better than  (\ref{yn2.est}), because $ \sqrt{A_1} \prec A_3$.
 Indeed, let us  estimate the term $A_3$.  We know that
\[
		A_3 = |\EE (Y_n^3)| =   n^{-\frac{3}{2}}\left|\sum_{l_1, l_2, l_3=1}^n \EE \left( \prod_{i=1}^3   g(X_{l_i}) \right) \right| .		\]
		Using the representation $g(X_{l_1}) =\delta ^2( g_2(X_{l_1}) l_1^{\otimes 2})$ and applying twice the duality relationship (\ref{dua}), we obtain
		\begin{eqnarray*}
		A_3& \leq & C n^{-\frac{3}{2}} \sum_{l_1, l_2, l_3=1}^n  \Big( | \EE ( g_2(X_{l_1}) g''(X_{l_2}) g(X_{l_3}) )| \rho_{12}^2 \\
		&& +2 |\EE ( g_2(X_{l_1}) g'(X_{l_2}) g'(X_{l_3}) ) \rho_{12} \rho_{13} |+
		| \EE ( g_2(X_{l_1}) g(X_{l_2}) g''(X_{l_3}) )| \rho_{13}^2 \Big).
		 \end{eqnarray*}
		Because $g$ is three times differentiable, we can still use the representations  $g(X_{l_3}) =\delta ( g_1( X_{l_3}) l_3)$,
		 $g'(X_{l_2})=\delta ( T_1(g')( X_{l_2}) l_2)$
		 and 
		$g(X_{l_2}) =\delta ( g_1( X_{l_2}) l_2)$,  and apply the duality relationship (\ref{dua}) again to produce an additional factor of the form
		$|\rho_{13}|+| \rho_{23}|$ for the first term and $|\rho_{12}|+| \rho_{23}|$  for the second and third terms. In this way, we obtain
		\[
		A_3 \le C
		 n^{-\frac{3}{2}} \sum_{l_1, l_2, l_3=1}^n \Big( |\rho_{12}^2 \rho_{13}| + |\rho_{12}\rho_{13}\rho_{23}|\Big) \,.
\]
  We make the  change of variables $\rho_{12}= \rho(k_1) $, $  \rho_{13} = \rho(k_2)$ and apply  (\ref{equ6})  with $M=2$  to the second summand  to obtain $$A_3 \leq C n^{-\frac{1}{2}} \sum_{|k|\leq n} |\rho(k)| + C n^{-\frac{1}{2}} \left(\sum_{|k|\leq n} |\rho(k)|^{\frac{3}{2}}\right)^2 \,.$$
Clearly, by \eqref{equ7}, this bound is not better than the bound we have previously obtained for $\sqrt{A_1}$, and (\ref{equ1}) gives the result in this case.

\medskip
\noindent{\it Case $g \in \D^{4,4}(\R,\gamma)$}. \quad As before, let us first estimate  the term $A_1$. Taking into account that $g$ has four derivatives, by the results of Lemma \ref{var.yn} and Lemma \ref{cov.pth1} and using  the notation $\rho(l_i - l_j) = \rho_{ij}$, we have
\[
    	A_1  \leq  \frac{C}{n^2} \sum_{l_1, l_2, l_3, l_4 = 1}^n |\rho_{12} \rho_{34} \rho_{24}|
		 \left((|\rho_{12}|+ |\rho_{14}| )\sum_{j \neq 3}| \rho_{j3}|  + |\rho_{13}| \right) \,.
\]
 We further write
    \begin{eqnarray}
    	A_1 & \leq & \frac{C}{n^2} \sum_{1\le l_i \le n, 1\le i \le 4}  \Big( \rho_{12}^2 \rho_{34}^2  |\rho_{24} |+ \rho_{12}^2 |\rho_{34} \rho_{24} \rho_{13}| + \rho_{12}^2| \rho_{34} \rho_{24} \rho_{23}|  +| \rho_{12} \rho_{34}^2\rho_{24} \rho_{14}| \nonumber \\
		& & \ + | \rho_{12} \rho_{34} \rho_{24} \rho_{14} \rho_{23}| +| \rho_{12} \rho_{34} \rho_{24} \rho_{14} \rho_{13}| + |\rho_{12}\rho_{34}\rho_{24}\rho_{13}| \Big) \label{a1.d2}\\
		& \leq & \frac{C}{n^2}   \sum_{1\le l_i \le n, 1\le i \le 4} \rho_{12}^2 \rho_{34}^2 |\rho_{24}| +  \rho_{12}^2 |\rho_{34} \rho_{24} \rho_{23}| +  |\rho_{12}\rho_{34}\rho_{24}\rho_{13}| \nonumber \,.
    \end{eqnarray}
For the second inequality in  \eqref{a1.d2}, we have used that the third and fourth summands are equal  and the fact that 
 $|\rho_{ij}|\le 1$. By a change of variables, we obtain
     \begin{eqnarray}
     	A_1 & \leq & \frac{C}{n} \sum_{|k_i| \leq n , 1\le i \le 3}  \Big( \rho^2(k_1) \rho^2(k_2)| \rho(k_3) |+ \rho^2(k_1) |\rho(k_2) \rho(k_3) \rho(k_2 - k_3) |\nonumber\\
		& & \ \ + | \rho(k_1) \rho(k_2) \rho(k_3) \rho(k_1 - k_2 + k_3)| \Big)   \label{a1.c4f}\,.
     \end{eqnarray}
     Using condition (\ref{h1}) and applying  inequality  (\ref{equ21})   with $M=2$ to handle the second summand and
     inequality (\ref{equ6}) with  $M=3$ for the  third summand, yields
    \begin{equation}\label{a1.d24} 
	 A_1 \leq  \frac Cn \sum_{|k| \leq n} |\rho(k)| + \frac Cn  \left(\sum_{|k| \leq n} |\rho(k)|^{\frac{4}{3}}\right)^3 \,.
	\end{equation}
This gives the desired estimate in view of   \eqref{yn2.est}.

As in the previous cases, we will show that, even with stronger integrability assumptions,
 using  either  (\ref{dist.ynz2}) or 
         (\ref{dist.ynz}) does not improve the above rate.
For this, consider first  the term $A_2$, assuming $g\in \D^{4,6} (\R,\gamma)$. Because $g$ has four derivatives, we can apply twice the duality relationship (\ref{dua}).
Recall that the term $A_2$ is bounded by \eqref{est.a2} and it suffices to consider the first summand in the right-hand side of this inequality. We write it here for convenience
\begin{equation}\label{est.a2.exp}
 A_{21}:=  \frac{2}{n^3} \sum_{l_1, \ldots, l_6=1}^n \EE (f_{1}(l_1, l_2, l_3) f_{1}(l_4, l_5, l_6)) \rho_{12} \rho_{13} \rho_{45} \rho_{46},
\end{equation}
where  $f_{1}(l_1, l_2, l_3) $ has been defined in (\ref{f1}).  Notice that the functions $g'$ and  $g_1$ have Hermite rank $1$.    We first write $g'(X_{l_2})=\delta( T_1(g')(X_{l_2}) l_2)$ and apply duality with respect to this divergence producing factors of the form $\rho_{2i}$, $i\not =2$, $1\le i \le 6$. Next we choose another function that has Hermite rank $1$ among the factors  $g_1(X_{l_3})$, $g'(X_{l_5})$ and $ g_1(X_{l_6})$, write it as a divergence integral and  apply duality again to obtain:
     \begin{eqnarray}\label{est.a22_4}
     	|\EE (f_{1}(l_1, l_2, l_3) f_{1}(l_4, l_5, l_6))| \leq C \sum_{i=1 \atop i \not =2}^6
	 \sum_{s\in\{3,5,6\}  \atop s\not =i} \sum_{j=1 \atop j\not =s }^6
	  |\rho_{2i} \rho_{sj}|.
     \end{eqnarray}
Applying inequality (\ref{equ9}) in  Lemma \ref{a3.c4} yields 
      \begin{equation} \label{equ40}
   A_2 \le 2   A_{21} \leq  \frac Cn \left(\sum_{|k| \leq n} |\rho(k)|\right)^2 \,.
      \end{equation}
	 By the inequality (\ref{equ7}) with $M=3$, we get that $A_1 \prec A_2$.

Next we will compare this estimate with the bound we can obtain for the term $A_3$ using the fact that $g$ has four derivatives. We can write
	\begin{eqnarray}
		A_3 &= & | \EE( Y_n^3)| =  C n^{-\frac{3}{2}}\left|\sum_{l_1, l_2, l_3=1}^n \EE \left( \prod_{i=1}^3  g(X_{l_i}) \right) \right| \nonumber\\
		& \leq &  C n^{-\frac{3}{2}} \sum_{l_1, l_2, l_3=1}^n   \Big(  \rho_{12}^2(|\rho_{13}|+|\rho_{23} |)^2+     |\rho_{12}\rho_{13}|(|\rho_{23}|+|\rho_{12}| (|\rho_{13}| + |\rho_{23}|))  \nonumber \\
		 &&+ \rho_{13}^2( |\rho_{12}| +|\rho_{23}| ) ^2 \Big) \nonumber\\
		& \leq &  C n^{-\frac{3}{2}} \sum_{l_1, l_2, l_3=1}^n \Big( |\rho_{12}^2 \rho^2_{13}| + |\rho_{12}\rho_{13}\rho_{23}| \Big) \label{a3.c4f} \,.
	\end{eqnarray}
Note that $n^{-\frac{3}{2}} \sum_{l_1, l_2, l_3=1}^n |\rho_{12}^2 \rho^2_{13}| = C n^{-\frac{1}{2}}$. We make the change of variables $\rho_{12}\to \rho(k_1), \rho_{13} \to \rho(k_2)$ and apply  (\ref{equ6}) to the second summand,  to obtain
\begin{equation}\label{a3.d24}
	A_3 \leq C n^{-\frac{1}{2}} \left(\sum_{|k|\leq n} |\rho(k)|^{\frac{3}{2}}\right)^2 \,.
\end{equation}
By (\ref{ho1})  with $M=3$ and   (\ref{ho2}), we obtain that $A_1\prec A_3$. By \eqref{ho3}, we have $A_3 \prec \sqrt{A_1}$. However, we cannot use the bound \eqref{dist.ynz} since the relationship between $\sqrt{A_1}$ and $A_2$ is not clear, because the sequences $n^{-\frac{1}{2}} (\sum_{|k| \leq n} |\rho(k)|)^{\frac{1}{2}}$ and $n^{-1} (\sum_{|k| \leq n}|\rho(k)|)^2$ are not comparable. An example could be $\rho(k) \sim k^{-\alpha}$ for $\alpha \in (\frac 12, \frac 23)$. So, we  use the bound \eqref{yn2.est} that is given by  (\ref{a1.d24}).

%=====================%=====================%======================
%===================== C~2,5+ Case=================================
%=====================%=====================%======================
\medskip
\noindent{\it Case $g \in  \D^{5,6} (\R,\gamma)$}.  \quad
 For the terms $A_1$ and $A_3$ we still have  the estimates \eqref{a1.d24} and \eqref{a3.d24}. 
  For the term $A_2$, we continue with the inequalities \eqref{est.a2.exp} and \eqref{est.a22_4}, and apply the duality for the third time to $\EE (f_{1}(l_1, l_2, l_3) f_{1}(l_4, l_5, l_6))$ when there is a factor with Hermite rank $1$, to obtain
\[
    	|\EE (f_{1}(l_1, l_2, l_3) f_{1}(l_4, l_5, l_6))| \leq C \sum_{\substack{i \neq s \neq j \\ i,s,j \in \{3,5,6\}}}|\rho_{2i} \rho_{sj}| 
		  + \ C \sum_{(i,s,j,t,h) \in D_3} |\rho_{2i} \rho_{sj} \rho_{th}| \,,
\]
    where
    \begin{equation} \label{d3}
    D_3= \{ (i,s,j,t,h): j,h \in \{1,\dots, 6\}; s,t \in \{3,5,6\};  i\not =2,  s\not \in \{ i,j\};  t  \not \in \{i,j,h\} \}.
    \end{equation}
By  inequality (\ref{equ10}) in  Lemma \ref{a3.c4},	  
     \begin{equation}  \label{equ3} 
     A_2 \leq \frac Cn \sum_{|k| \leq n} |\rho(k)| +  \frac Cn  \left(\sum_{|k|\leq n} |\rho(k)|^{\frac{3}{2}}\right)^4 \,.
     \end{equation}
     From (\ref{a1.d24}), (\ref{equ3}) and (\ref{ho3}) we deduce that  $A_2 \prec A_1$  and, therefore, $ A_1+ \sqrt{A_2} \prec \sqrt{A_1}$.  Therefore,  (\ref{dist.ynz2}) gives a better rate than (\ref{yn2.est}), which is given by
     \begin{equation} \label{equ4}
     A_1 + \sqrt{A_2} \le C n^{-\frac 12 } \left(\sum_{|k| \leq n} |\rho(k)| \right)^{\frac 12} + C  n^{-\frac 12} \left(\sum_{|k|\leq n} |\rho(k)|^{\frac{3}{2}}\right)^2.
     \end{equation}
     
     Clearly, $A_3 \prec   A_1 + \sqrt{A_2} $. 
Whether we choose \eqref{dist.ynz2} or \eqref{dist.ynz} depends on the computation of $A_4$, where we need to assume $g\in \D^{5,8}(\R,\gamma)$. 
Consider the second summand in the expression (\ref{equ90}) denoted by
\begin{eqnarray}
(A_{42})^2 & :=& \frac{2}{n^4} \sum_{l_j=1, j=1, \ldots, 8}^n \sum_{i=1}^3
 \EE \left(f_{2}^{(i)}(l_1, l_2, l_3)  g_1(X_{l_4}) f_{2}^{(i+4)}(l_5, l_6, l_7)  g_1(X_{l_8})\right)  \nonumber  \\
&&   \times \rho_{12} \rho_{23}\rho_{i4}  \rho_{56} \rho_{67}\rho_{(i+4) 8} \,. \label{est.a4}
\end{eqnarray}
Taking into account that $g$ has five derivatives and the  terms $f_2^{(2)}$ and $f_2^{(6)}$ involve $g'''$,
we can apply duality  twice  using the factors that have Hermite rank $1$.  In this way, we  get the following item in the bound of $A_{42}$:
  \[
  \sqrt{\frac{C}{n^4} \sum_{|l_j|=1, j=1, \ldots, 8}^n \rho_{12}^2 \rho_{13} \rho_{24} \rho_{56}^2 \rho_{67} \rho_{68}}, 
  \]
which gives the rate  $\frac{1}{n} \left(\sum_{|k| \leq n} |\rho(k)|\right)^2$. This rate cannot always be better than that of $A_1+\sqrt{A_2}$ bound since 
 the sequences  $\frac{1}{n} \left(\sum_{|k| \leq n} |\rho(k)|\right)^2$ and $n^{-\frac{1}{2}} \left(\sum_{|k| \leq n} |\rho(k)|^{\frac{3}{2}} \right)^2$ are not comparable. An example could be $\rho(k) \sim k^{-\alpha}$ for $\alpha \in (\frac{1}{2}, \frac{2}{3})$.
This suggests us using the bound \eqref{dist.ynz2} that is given by  (\ref{equ4}).

    \medskip
\noindent{\it Case $g \in \D^{6,8}(\R,\gamma)$}.  \quad
For the terms $A_1$, $A_2$  and $A_3$, we still have  the estimates \eqref{a1.d24}, \eqref{equ3}  and \eqref{a3.d24}. 
     Let us now study the term  $A_4$ given by (\ref{equ90}).           The terms $f_2^{(2)}$ and $f_2^{(6)}$ involve $g'''$ and they can be integrated by parts three times. Therefore, we are going to use only three integration by parts. On the other hand, the terms $f_2^{(2)}$,  $f_2^{(6)}$ , $f_1^{(1)}$ and $f_1^{(4)}$ have two factors with Hermite rank one that can be represented as divergences, but the other terms have only one.  All these terms are similar, with the only difference being the number of factors with Hermite rank one.  We will handle only   the term $f_1^{(1)}$ that has two factors with Hermite rank one and the term   $f_1^{(2)}$ that has only one.  The other terms could be treated in a similar way.
      In this way, for the term 
     $f_1^{(1)}$, we  obtain, after integrating by parts three times,
\[
 \left|	\EE \left(f_{1}^{(1)}(l_1, l_2, l_3) g_1(X_{l_4}) f_{1}^{(5)}(l_5, l_6, l_7)  g_1(X_{l_8})\right)  \right|
	\leq C \sum_{ (i,s,j,t,h) \in D_4} |\rho_{2i} \rho_{sj} \rho_{th}|,
\]
where 
\begin{equation}  \label{d4}
D_4=  \left\{ (i,s,j,t,h):  1\le i,j,h \le 8; s,t \in \{ 3,4,6,7,8\};  i \not =2;  s \not \in \{i,j\};  t \not \in \{i,s,j,h\} \right\}.
\end{equation}
On the other hand, for the term     $f_1^{(2)}$, we  obtain, after integrating by parts three times,
\begin{eqnarray*}
	&& \left| \EE \left(f_{1}^{(2)}(l_1, l_2, l_3)  g_1(X_{l_4}) f_{1}^{(6)}(l_5, l_6, l_7)  g_1(X_{l_8})\right) \right| \\
	&\leq& C \sum_{\substack{i \neq s \neq j\\ i, s, j\in\{4,7,8\}}}  |\rho_{3i} \rho_{sj}|  +  C \sum_{ (i,s,j,t,h) \in D_5} |\rho_{3i} \rho_{sj} \rho_{th}|,
\end{eqnarray*}
where  
\begin{equation}  \label{d5}
D_5=  \left\{ (i,s,j,t,h):  1\le i,j,h \le 8; s,t \in \{ 4,7,8\};  i \not =3;  s \not \in \{i,j\};  t \not \in \{i,s,j,h\} \right\}.
\end{equation}
By Lemma  \ref{a4.c6} and Lemma \ref{a4.c6-2}, we obtain
\[
A_4 \le 
  \frac{C}{n} \left(\sum_{|k| \leq n} |\rho(k)|\right)^{\frac 32}  +\frac{C}{n} \left(\sum_{|k|\leq n}|\rho(k)|^{\frac{4}{3}}\right)^3.
  \]
  Then, from (\ref{ho1}) with $M=3$ and (\ref{ho2}), we deduce $A_4 \prec A_3$.
  We already know that $A_2 \prec A_1 \prec A_3 \prec \sqrt{A_2}$. Also  using (\ref{ho3}) it follows that $A_3 \prec \sqrt{A_1}$. 
Thus, we use \eqref{dist.ynz} for the bound of $d_{\rm TV}(Y_n/\sigma_n, Z)$ which is given by  the estimate \eqref{a3.d24} of the term $A_3$. \\
\end{proof}

\subsection{Case  $d \geq 3$}

\begin{thm} \label{thm3.4}
   Assume $g \in \D^{3d-2,4}(\R,\gamma)$ has Hermite rank $d\ge 3$ and  suppose that (\ref{h1}) holds true. Then we have the following estimate
		   \begin{eqnarray} 
		   d_{\rm TV} (Y_n / \sigma_n ,Z) \le C n^{-\frac{1}{2}} \sum_{|k| \leq n} |\rho(k)|^{d-1} \left(\sum_{|k| \leq n} |\rho(k)|^{2}\right)^{\frac{1}{2}} \nonumber\\
		   + C n^{-\frac{1}{2}} \left(\sum_{|k|\leq n}|\rho(k)|^2\right)^{\frac{1}{2}} \left(\sum_{|k|\leq n}|\rho(k)|\right)^{\frac{1}{2}}\label{equ70a}\,.
\end{eqnarray}
\end{thm}

\begin{proof}[Proof of Theorem \ref{thm3.4}]
Inequality (\ref{equ70a})  will be established using Proposition \ref{stein.hd0} that is specifically expressed as \eqref{yn2.est}.
The proof will be done in two steps.

\medskip
  \noindent{\it Step 1:}  First, we consider  the case when  $g$ is the Hermite polynomial $H_d$.
	By  Lemma \ref{var.yn} and Lemma \ref{cov.pth2}, we have
	\[
	A_1\le  \frac{C}{n^2} \sum_{l_1, l_2, l_3, l_4=1}^n 
	|\rho(l_1 - l_2)^{\beta_1} \rho(l_3 - l_4)^{\beta_2} \rho(l_2 - l_4)^{\beta_3} \rho(l_1 - l_3)^{\beta_4} \rho(l_1 - l_4)^{\beta_5} \rho(l_2 - l_3)^{\beta_6} |,
	\]
	where  the $\beta_i$'s satisfy $\sum_{i=1}^6 \beta_i = 2d$, $\beta_2 + \beta_3 + \beta_5 = d$, $\beta_1 + \beta_3 + \beta_6 = d$, $\beta_1 + \beta_4 + \beta_5 = d$, $\beta_2 + \beta_4 + \beta_6 = d$ and  $\beta_j \geq 1$ for $j=1,2,3$. Making the change of variables, $l_i - l_4 \to k_i$, $i=1,2,3$ yields
	\[
	A_1 \le \frac{C}{n} \sum_{k_1, k_2, k_3=1}^n |\rho(k_1 - k_2)^{\beta_1} \rho(k_3)^{\beta_2} \rho(k_2)^{\beta_3} \rho(k_1 - k_3)^{\beta_4} \rho(k_1)^{\beta_5} \rho(k_2 - k_3)^{\beta_6} |\,.
	\]
	Applying the Brascamp-Lieb inequality (\ref{BL}), we can write
	   \[
	   A_1 \leq  \frac Cn  \prod_{i=1}^6 \left(\sum_{|k_i| \leq n} |\rho(k_i)|^{\frac{\beta_i}{p_i}}\right)^{p_i},
	   \]
where the $p_i$'s satisfy $     \sum_{i=1}^6 p_i = 3$, $ p_i \leq 1$,  $p_1 + p_3 + p_5 \leq 2$,  $ p_2 + p_3 + p_6 \leq 2$, $ p_2 + p_4 + p_5 \leq 2$ and $p_1 + p_4 + p_6 \leq 2$. 
The restriction of $\beta_i$ could be further simplified as
    $$\beta_1 = \beta_2, \beta_3 = \beta_4, \beta_5 = \beta_6, \beta_1 + \beta_3 + \beta_5 = d, \ {\rm and} \ \beta_1, \beta_3 \geq 1 \,.$$
Then we choose $p_1 = p_2, p_3 = p_4, p_5 = p_6$ to obtain
    \begin{equation}\label{a1d}
		A_1 \leq  \frac Cn  \left(\prod_{i=1,3,5} \left(\sum_{|k_i| \leq n} |\rho(k_i)|^{\frac{\beta_i}{p_i}}\right)^{p_i} \right)^2 \,.
	\end{equation}
	We are going to choose $p_i = \frac{\beta_i}{d-1} + \epsilon_i$ for $i=1,3,5$, where the  $\epsilon_i$'s  satisfy
	$\epsilon _i  \ge 0$ and  $\frac{d}{d-1} + \sum_{i=1,3,5} \epsilon_i = \frac 32$. To choose the values of the $\epsilon _i$'s we consider two cases. Set  $\delta = \frac{1}{2} - \frac{1}{d-1}$.
	\begin{itemize}
	 \item[(i)]  Suppose that $\delta \le 1- \frac {\beta_1} {d-1}$. Then, we take $\epsilon_1=\delta$ and  $\epsilon_3 = \epsilon_5 = 0$ and we obtain  $p_1= \frac {\beta_1} {d-1} +\frac 12 -\frac 1{d-1}$, $p_3=\frac{\beta_3} {d-1}$ and $p_5=\frac{\beta_5} {d-1}$. 
	 \item[(ii)] Suppose that $\delta \ge 1- \frac {\beta_1} {d-1}$. Then, we take $\epsilon_1= 1-\frac {\beta_1} {d-1} $ and  $\epsilon_3 = 
	 \delta - \epsilon_1$ and $\epsilon_5 = 0$ and we obtain
	   $p_1=1$, $p_3=\frac{\beta_3} {d-1}+\frac{\beta_1} {d-1} -\frac 12  -\frac 1{d-1}$ and $p_5=\frac{\beta_5} {d-1}$. 
\end{itemize}
	It is easy to show that these $p_i$'s satisfy the desired conditions and, furthermore,
 $\beta_i \geq 2 p_i$ for $i=1,3,5$.  This allows us to choose the pairs $(\alpha_i, \gamma_i)$ that satisfy the following equations 
\begin{equation}\label{ab.eqs}\frac{\alpha_i}{2} + \frac{\gamma_i}{d-1} = 1, \ {\rm and} \ \alpha_i + \gamma_i = \frac{\beta_i}{p_i} \,.
\end{equation} 
Then H\"older inequality implies
    $$\sum_{|k| \leq n} |\rho(k)|^{\frac{\beta_i}{p_i}} \leq \left(\sum_{|k| \leq n} |\rho(k)|^2\right)^{\frac{\alpha_i}{2}} \left(\sum_{|k| \leq n} |\rho(k)|^{d-1}\right)^{\frac{\gamma_i}{d-1}} \,.$$
Then we plug this inequality into \eqref{a1d} and solve $\alpha_i, \gamma_i$ from \eqref{ab.eqs}. In this way, we obtain the inequality 
	\begin{equation} \label{a1.equ4}
		A_1 \le \frac Cn \left(\sum_{|k| \leq n} |\rho(k)|^{d-1}\right)^2 \sum_{|k| \leq n} |\rho(k)|^{2} \,.
	\end{equation}

%\noindent {\bf Case 2:} $\alpha > \frac{1}{d-1}$. It suffices to consider the case $\beta_1 \geq \beta_3 \geq \beta_5$. Then we choose $p_i = \frac{\beta_i}{d-1}$ for $i=1, 5$ and $p_3 = \frac{3}{2} - \frac{d-\beta_3}{d-1}$. Then one can show that $p_3 \leq 1$ and $f_n(p_1, p_2, p_3) \leq (\frac{1}{2} - \alpha)_+$.\\

\noindent {\it Step 2:} We consider the case  $g \in \D^{3d-2}(\R,\gamma)$. By  Lemma \ref{var.yn} and Lemma \ref{cov.pth2}, we have
	\begin{equation} \label{equ96}
	A_1 \leq \frac{C}{n^2} \sum_{l_1, l_2, l_3, l_4=1}^n |\rho(l_1 - l_2)^{\beta_1} \rho(l_3 - l_4)^{\beta_2} \rho(l_2 - l_4)^{\beta_3} \rho(l_1 - l_3)^{\beta_4} \rho(l_1 - l_4)^{\beta_5} \rho(l_2 - l_3)^{\beta_6}| ,
	\end{equation}
	where  the $\beta_i$'s satisfy $\beta_i \leq d$, $\beta_j \geq 1$ for $j=1,2,3$, $ \sum_{i=1}^6 \beta_i \leq 3d-1$
	and the lower bounds
	\begin{eqnarray*}
 \beta_2 + \beta_3 + \beta_5  &\geq &d, \\
 \beta_1 + \beta_3 + \beta_6  &\geq &d, \\
 \beta_1 + \beta_4 + \beta_5 &\geq &d, \\
 \beta_2 + \beta_4 + \beta_6  &\geq &d. 
 \end{eqnarray*}
    When all the above $\beta_i$'s inequalities attain the lower bound $d$, the right hand-side of (\ref{equ96}) coincides with the case when $g$ is the Hermite polynomial $H_d$. This case has been discussed in Step 1.  
    On the other hand,  if $\beta_1 \wedge \beta_2 + \beta_3 \wedge \beta_4 + \beta_5 \wedge \beta_6 \geq d$ and $\beta_3 \wedge \beta_4 \geq 1$, taking into account that $|\rho|\leq 1$, the right-hand side of (\ref{equ96}) is actually dominated by the case where all the  $\beta_i$'s inequalities attain the lower bound $d$.

    Now we need to consider the all the other possible cases. 
    In each case, we make the change of variables $l_1 - l_2=k_1, l_3 - l_4 = k_2, l_2 - l_4 = k_3 $.

\medskip
\noindent
{\it (i) Case $\beta_4 = \beta_5 = \beta_6=0$}. \quad Then $\beta_1 = \beta_2 =d$, $\beta_3 = 1$.  For these values of the $\beta_i$'s we can write the right hand-side of (\ref{equ96}) as 
	\begin{eqnarray*}
	\frac{1}{n^2}\sum_{l_1, l_2, l_3, l_4=1}^n |\rho(l_1 - l_2)^d \rho(l_3 - l_4)^d \rho(l_2 - l_4)|  &=& \frac{1}{n} \sum_{|k_i|\leq n, 1\le i\le 3}|\rho(k_1)|^d |\rho(k_2)|^d |\rho(k_3)| \\
	&\leq &\frac{C}{n} \sum_{|k|\leq n} |\rho(k)| \,.
	\end{eqnarray*}
  
  \medskip
\noindent
{\it (ii) Case
 $\beta_4 = \beta_5 = 0$, $\beta_6 > 0$.} \quad  Then $\beta_1 = d, \beta_2<d$, $\beta_2 + \beta_3 \geq d$ and $\beta_2 + \beta_6 \geq d$. Using \eqref{h1}, we can write
    \begin{eqnarray*}
    	A_1 &\leq& \frac Cn \sum_{|k_i| \leq n, i = 2,3} |\rho(k_2)|^{\beta_2} |\rho(k_3)|^{\beta_3} | \rho(k_3 - k_2)|^{\beta_6} \\
		 &\leq& \frac Cn \sum_{|k_i| \leq n, i = 2,3} |\rho(k_2)|^{\beta_2} |\rho(k_3)|^{d - \beta_2} | \rho(k_3 - k_2)|^{d-\beta_2} \\
		 &  \leq & \frac Cn \sum_{|k| \leq n} |\rho(k)|^{d - \beta_2} \leq \frac Cn \sum_{|k| \leq n} |\rho(k)|,
    \end{eqnarray*}
	where in the third inequality we have used \eqref{BL} with $p_1 = \frac{\beta_2}{d}$, $p_2 = 1$ and $p_3 = \frac{d-\beta_2}{d}$.

  \medskip
\noindent
{\it (iii) Case $\beta_4 = \beta_6 = 0, \beta_5 > 0$.} \quad  This case is similar to (ii).

  \medskip
\noindent
{\it (iv) Case $\beta_5 = \beta_6 = 0, \beta_4 > 0$. } \quad Then $\beta_2 + \beta_3 \geq d$, $\beta_1 + \beta_3 \geq d$, $\beta_1 + \beta_4 \geq d$, $\beta_2 + \beta_4 \geq d$. It is easy to see $\beta_1 \wedge \beta_2 + \beta_3 \wedge \beta_4 + \beta_5 \wedge \beta_6 \geq d$  and, furthermore, $\beta_3 \wedge \beta_4 \ge 1$.  This situation  has been discussed before  and $A_1$  is dominated by the bound in the case where $g$ is the Hermite polynomial.

  \medskip
\noindent
{\it (v)   $\beta_4 = 0, \beta_5 > 0, \beta_6 > 0$.} \quad  Then $\beta_1<d$, $\beta_2<d$, $\beta_1 + \beta_5 \geq d, \beta_2 + \beta_6 \geq d$. As a consequence, we obtain 
\begin{eqnarray*}
	A_1 & \leq & \frac Cn \sum_{|k_i| \leq n, 1 \leq i \leq 3} |\rho(k_1)|^{\beta_1} |\rho(k_2)|^{\beta_2} |\rho(k_3)|^{\beta_3} |\rho(k_1 + k_3)|^{d- \beta_1} |\rho(k_3 - k_2)|^{d-\beta_2} \\
	& \leq & \frac Cn \sum_{|k| \leq n} |\rho(k)|,
\end{eqnarray*}
  where have used \eqref{BL} for $p_i = \frac{\beta_i}{d}$ for $i=1,2$, $p_3 = 1$ and $p_{i+3} = \frac{d-\beta_i}{d}$ for $i=1,2$.
  
    \medskip
\noindent
{\it (vi)  $\beta_5=0, \beta_4> 0, \beta_6 > 0$.} \quad  Then $\beta_2 + \beta_3 \geq d, \beta_1 + \beta_4 \geq d$. This case is similar to (v).

    \medskip
\noindent
{\it (vii)  $\beta_6=0, \beta_4> 0, \beta_5 > 0$.} \quad  This case is similar to (v) and (vi).

    \medskip
\noindent
{\it (viii)   $\beta_i > 0$ for all $1 \leq i \leq 6$, and $\beta_1 \wedge \beta_2 + \beta_3 \wedge \beta_4 + \beta_5 \wedge \beta_6 < d$. } \quad Without loss of generality, we may assume that $\beta_1 \leq \beta_2$. We take into account of $\beta_1 + \beta_4 + \beta_5 \geq d$ and $\beta_1 + \beta_3 + \beta_6 \geq d$, so there are two cases: $\beta_3 \leq \beta_4, \beta_5 \leq \beta_6$; and $\beta_4 \leq \beta_3, \beta_6 \leq \beta_5$. These two cases are actually equivalent, because in the second case, we can make the change of variable $l_3 - l_1 \to k_3$, instead of $l_2 - l_4 \to k_3$ for the first case.  Thus it sufficies to consider the first case, i.e.,
  \begin{eqnarray*}
  	A_1 \leq \frac Cn \sum_{\substack{|k_i| \leq n, \\ 1 \leq i \leq 3}} |\rho(k_1)|^{\beta_1} |\rho(k_2)|^{\beta_2} |\rho(k_3)|^{\beta_3} |\rho(k_1 - k_2 + k_3)|^{\beta_4} |\rho(k_1 + k_3)|^{\beta_5} |\rho(k_3 - k_2)|^{\beta_6},
  \end{eqnarray*}
  where $\beta_1 + \beta_3 + \beta_5 < d$, $\beta_2 + \beta_4 + \beta_6 > d$ since $\sum_{i=1}^6 \beta_i  \ge 2d$.
  
  Next we will apply Brascamp-Lieb inequality \eqref{BL} according to several different subcases.
  \begin{itemize}
	  \item [(1)] Suppose $\beta_1 \wedge \beta_3 \wedge \beta_5 = \beta_1$. Then if $\sum_{i=2}^6 \beta_i \geq 2d$,  the right-hand side of the above inequality  is bounded by the case $\sum_{i=2}^6 \beta_i = 2d$ when we decrease $\beta_i$'s, $i=2,4,6$ appropriately. We use \eqref{BL} with $p_1 = 1$, $p_i = \frac{\beta_i}{d}$ for $i\geq 2$, taking into account that $|\rho| \leq 1$, to obtain
    \begin{eqnarray*}
    	A_1 \leq \frac Cn \sum_{|k| \leq n} |\rho(k)|^{\beta_1} \leq \frac Cn \sum_{|k| \leq n} |\rho(k)|\,.
    \end{eqnarray*}
  If $\sum_{i=2}^6 \beta_i < 2d$, for which an example could be $\beta_1 = 2, \beta_3 = 2, \beta_5 = d-5, \beta_2 = 3, \beta_4 = 3, \beta_6 = d-4$, then taking into account $|\rho| < 1$, we obtain
    \begin{eqnarray*}
    	A_1 &\leq& \frac Cn \sum_{\substack{|k_i| \leq n, \\ 1 \leq i \leq 3}} |\rho(k_1)|^{\beta_1} |\rho(k_2)|^{\beta_2} |\rho(k_3)|^{\beta_3} \\
		 & & \qquad \qquad |\rho(k_1 - k_2 + k_3)|^{\beta_4} |\rho(k_1 + k_3)|^{\beta_5} |\rho(k_3 - k_2)|^{\beta_6}
    \end{eqnarray*}
	where $\beta_2 + \beta_4 + \beta_6 = d$ and also $\beta_1 + \beta_3 + \beta_5 < d$. Applying \eqref{BL} with $p_i = \frac{\beta_i}{\beta_1 + \beta_3}$ for $i=1,3$, $p_5 = 1$ and $p_i = \frac{\beta_i}{d}$ for $i=2, 4, 6$, we obtain
	 \begin{eqnarray*}
	 	A_1 \leq \frac Cn \sum_{|k| \leq n} |\rho(k)|^{\beta_1 + \beta_3} \sum_{|k| \leq n}|\rho(k)|^{\beta_5} \leq \frac Cn \sum_{|k| \leq n} |\rho(k)|^2 \sum_{|k| \leq n}|\rho(k)|.
	 \end{eqnarray*}
	 \item [(2)] $\beta_1 \wedge \beta_3 \wedge \beta_5 = \beta_5$. We use the same approach as for the subcase (1).
	 \item [(3)] $\beta_1 \wedge \beta_3 \wedge \beta_5 = \beta_3$. We follow  the same methodology. When $\sum_{i \neq 3} \beta_i < 2d$, the arguments are the same. When $\sum_{i \neq 3} \beta_i \geq 2d$, since $d \leq \beta_1 + \beta_4 + \beta_5 < 2d$, we can decrease $\beta_2, \beta_6$ appropriately such that $\sum_{i \neq 3} \beta_i = 2d$ and at the same time this implies $\beta_2 + \beta_6 \leq d$. Then we use \eqref{BL} with $p_3 = 1$, $p_i = \frac{\beta_i}{d}$ for $i\neq 3$ to obtain
	      \begin{eqnarray*}
	      	A_1 \leq \frac Cn \sum_{|k| \leq n} |\rho(k)|^{\beta_3} \leq \frac Cn \sum_{|k| \leq n} |\rho(k)|\,.
	      \end{eqnarray*}
   \end{itemize}

	This completes the proof of the theorem.
\end{proof}

\begin{remark}
In the case of the Hermite polynomial $g=H_d$, $d\ge 3$, the proof of Theorem \ref{thm3.4}, based on Proposition \ref{stein.hd0},  yields
     \begin{equation} \label{equ71}
		   d_{\rm TV} (Y_n / \sigma_n ,Z) \le C n^{-\frac{1}{2}} \sum_{|k| \leq n} |\rho(k)|^{d-1} \left(\sum_{|k| \leq n} |\rho(k)|^{2}\right)^{\frac{1}{2}}  \,.
\end{equation}
In this case  Proposition \ref{stein.hd}  reduces to the computation of the third and fourth cumulants and  one can derive
    the following bound (see \cite{bbnp}), which is better than
 (\ref{equ71}):
\[
		   d_{\rm TV} (Y_n / \sigma_n ,Z) \le  \frac Cn \left( \sum_{|k| \leq n} |\rho(k)|^{d-1} \right)^2 \sum_{|k| \leq n} |\rho(k)|^2
		   +  \frac{C}{\sqrt{n}} \left( \sum_{|k| \leq n} |\rho(k)|^{\frac {3d} 4} \right)^2 \mathbf{1}_{\{d \,\, {\rm even}\}} \,.
\]		   
However, applying Proposition \ref{stein.hd}  to the case of a general function $g$ is a much harder problem and it will not be dealt in this paper. 
\end{remark}

Consider the particular case where $\rho(k) \sim k^{-\alpha}$, as $k$ tends to infinity,  for some $\alpha>0$. Then, condition  (\ref{h1}) is satisfied provided $\alpha d >1$. In this case, Theorems \ref{main.d1}, \ref{main.d12} and \ref{thm3.4} imply the following results.

\begin{cor}  \label{cor1}
Suppose that  $\rho(k) \sim k^{-\alpha}$, as $k$ tends to infinity, where $\alpha>0$ is such that  $\alpha d >1$. Then, the following estimates hold true  in the context of Theorem \ref{bm}:
\begin{itemize}
\item[(i)]  If  $g \in \D^{2,4} (\R,\gamma)$ has Hermite rank $1$ and $\alpha > 1$, 
\[d_{\rm TV}(Y_n/\sigma_n, Z) \leq   Cn^{-\frac{1}{2}}  \,.
\]

\item[(ii)] If $g \in \D^{2,4} (\R,\gamma)$ has Hermite rank $2$ and $\alpha > \frac 23$,  
      \[
		        d_{\rm TV}(Y_n / \sigma_n, Z) \leq
		        \begin{cases}
		        Cn^{-\frac{1}{2}}  &  {\rm if}  \ \alpha > 1\,, \\
		         Cn^{-\frac{1}{2}} (\log n)^{\frac 32}  &   {\rm if}  \ \alpha =1\,, \\
		        Cn^{1-\frac 32 \alpha}  &   {\rm if}  \ \alpha   \in (\frac 23, 1). 
		        \end{cases}  
		        \]
	\item[(ii)]	   If $g \in \D^{3,4} (\R,\gamma)$ has Hermite rank $2$,
	      \[
		        d_{\rm TV}(Y_n / \sigma_n, Z) \leq
		        \begin{cases}
		        Cn^{-\frac{1}{2}}  &  {\rm if}  \ \alpha > 1\,, \\
		         Cn^{-\frac{1}{2}} \log n  &   {\rm if}  \ \alpha =1\,, \\
		        Cn^{\frac{1}{2} -\alpha}  &   {\rm if}  \ \alpha  \in(\frac 12, 1). 
		        \end{cases}  
		        \]
		\item[(iii)]	 If $g \in \D^{4,4} (\R,\gamma)$ has Hermite rank $2$,
		      \[
		        d_{\rm TV}(Y_n / \sigma_n, Z) \leq
		        \begin{cases}
		        Cn^{-\frac{1}{2}}  &  {\rm if}  \ \alpha > 1\,, \\
		         Cn^{-\frac{1}{2}}( \log n )^{\frac 12}  &   {\rm if}  \ \alpha = 1\,, \\
		        Cn^{ -\frac  \alpha 2}  &   {\rm if}  \ \alpha  \in(1, \frac 23) \, ,\\
		        Cn^{  1-2\alpha}  &   {\rm if}  \ \alpha  \in(\frac 12, \frac 23] \, .
		        \end{cases}  
		        \]
		\item[(iv)]	  If $g \in \D^{5,6} (\R,\gamma)$ has Hermite rank $2$,
		  \[
		        d_{\rm TV}(Y_n / \sigma_n, Z) \leq
		        \begin{cases}
		        Cn^{-\frac{1}{2}}  &  {\rm if}  \ \alpha > 1\,, \\
		         Cn^{-\frac{1}{2}}( \log n )^{\frac 12}  &   {\rm if}  \ \alpha = 1\,, \\
		        Cn^{-\frac  \alpha 2}  &   {\rm if}  \ \alpha  \in(\frac 35, 1) \, ,\\
		        Cn^{ \frac 32-3\alpha}  &   {\rm if}  \ \alpha  \in(\frac 12, \frac 35] \, .
		        \end{cases}  
		        \]
		  \item[(v)] 
		 If $g \in \D^{6,8} (\R,\gamma)$ has Hermite rank $2$,
			   \[
		        d_{\rm TV}(Y_n / \sigma_n, Z) \leq
		        \begin{cases}
		        Cn^{-\frac{1}{2}}  &  {\rm if}  \ \alpha > \frac 23\,, \\
		         Cn^{-\frac{1}{2}}( \log n )^2  &   {\rm if}  \ \alpha = \frac 23\,, \\
		        Cn^{ \frac 32-3 \alpha}  &   {\rm if}  \ \alpha  \in(\frac 12, \frac 23) \, .
		        \end{cases}  
		        \]
 \item[(vi)]
 If $g \in  \D^{3d-2,4}(\R,\gamma)$ has Hermite rank $d\ge 3$,
  \[
		        d_{\rm TV}(Y_n / \sigma_n, Z) \leq
		        \begin{cases}
		        Cn^{-\frac{1}{2}}  &  {\rm if}  \ \alpha > 1\,, \\
		         Cn^{-\frac{1}{2}}( \log n )^{\frac 12}  &   {\rm if}  \ \alpha = 1\,, \\
		        Cn^ { -\frac \alpha 2}  &   {\rm if}  \ \alpha  \in( \frac 12, 1) \,, \\
				 Cn^ { -\frac \alpha 2} \sqrt{\log n} &   {\rm if}  \ \alpha = \frac 12 \,, \\
				 Cn^{\frac 12 - \frac 32 \alpha} &   {\rm if}  \ \alpha \in (\frac{1}{2d-3}, \frac{1}{2}) \\
		          Cn^ {1 - \alpha d}  &   {\rm if}  \ \alpha  \in( \frac 1d, \frac 1 {2d-3}] \,.\\
		        \end{cases}  
		        \]
\item[(vii)] When $g=H_d$, $d\ge 3$, the bound \eqref{yn2.est} combined with the estimate \eqref{a1.equ4} yields
   \[
		        d_{\rm TV}(Y_n / \sigma_n, Z) \leq
		        \begin{cases}
		        Cn^{-\frac{1}{2}}  &  {\rm if}  \ \alpha > \frac 12\,, \\
		         Cn^{-\frac{1}{2}}( \log n )^{\frac 12}  &   {\rm if}  \ \alpha = \frac 12\,, \\
		        Cn^ { - \alpha }  &   {\rm if}  \ \alpha  \in( \frac 1 {d-1}, \frac 12) \,, \\
				  Cn^ { - \alpha } \log n &   {\rm if}  \ \alpha = \frac 1 {d-1} \,, \\
		          Cn^ {1 - \alpha d}  &   {\rm if}  \ \alpha  \in( \frac 1d, \frac   1{d-1}) \,.\\
		        \end{cases}  
		        \]
 \end{itemize}
 \end{cor}
 We remark that the bounds derived in point (vii) coincide with  the estimates obtained by Bierm\'e, Bonami and Le\'on in \cite{bbl} using  techniques of Fourier analysis. 
	Corollary \ref{cor1}  can be applied to any function $g$  with an expansion $g(x)= \sum_{m=d} ^ {d+k} c_m H_m(x) $ for any $k \geq 0$.

\section{Application to fractional Brownian motion}

Recall that the fractional Brownian motion (fBm) $B=\{B_t, t  \in \R\}$ with Hurst parameter $H \in (0,1)$ is a zero mean Gaussian process, defined on a complete probability space $(\Omega, \mathcal{F},P)$, with the covariance function 
\[
\EE(B_s B_t) = \frac{1}{2} (|s|^{2H} + |t|^{2H} - |s-t|^{2H}) \,.
\]
The fractional noise defined by $X_j = B_{j+1} - B_j$, $j \in \mathbb{Z}$ is an example of a Gaussian stationary sequence with unit variance. The covariance function is given by
\[
\rho_H(j) =  \frac 12 \left( |j+1|^{2H} + |j-1|^{2H}  -2|j|^{2H} \right).
\]
Notice that $\rho_H(j)$ behaves as  $H(2H-1) j^{2H-2}$ as $j\rightarrow \infty$. Thus, this covariance function  has a power decay at infinity with $\alpha = 2-2H$.
Consider the sequence
 $Y_n$  defined by 
 \[
 Y_n = \frac{1}{\sqrt{n}} \sum_{j=1}^n g(B_{j+1} - B_j)\,,
 \]
 where $g\in L^2(\R, \gamma)$ has Hermite rank $d\ge 1$. 
 As a consequence, the estimates obtained in Corollary \ref{cor1} hold with $\alpha =2-2H$.

\subsection{Application to  the asymptotic behavior of power variations}

For any $p\ge 1$, the power variation of the fBm on the time interval $[0,1]$ is given by 
\[
 V_n^p(B)= \sum_{j=0}^{n-1} \left|B_{\frac{j+1}{n}} - B_{\frac{j}{n}}\right|^p \,.
 \]
By the self-similarity property of fBm, the sequence $\{ n^{H} (B_{\frac{j+1}{n}} - B_{\frac{j}{n}}), j\ge 0\}$ has the same distribution as $\{B_{j+1} - B_j, j\ge 0\}$, which is stationary and ergodic. By the Ergodic Theorem, we have, as $n \to \infty$,
 \[
 n^{pH-1} V_n^p(B) \to c_p
 \]
 almost surely and in $L^q(\Omega)$ for any $q\ge 1$,   where  $c_p = \EE( |Z|^p)$. 
 Moreover, when $H \in (0, \frac{3}{4})$, using the fact that the   function $g(x)= |x|^p -c_p$ has Hermite rank $2$, the Breuer-Major theorem leads to the following central limit theorem
 \begin{equation}  \label{equ70}
 S_n:= \sqrt{n}\left(n^{pH-1} V_n^p(B) - c_p\right) \to   N(0, \sigma^2_{H,p}),
 \end{equation}
 where $\sigma^2_{H,p} = \sum_{m=2} ^\infty c_m^2 m! \sum_{k\in \mathbb{Z}} \rho_H(k) ^m$, with $|x|^p -c_p =\sum_{m=2} ^\infty c_m 
 H_m(x)$.
 A functional version of this central limit theorem can also be proved  (see \cite{cnw}).
 
 We can apply the results obtained in Section 3 to derive the rate of convergence for the total variation distance  in   (\ref{equ70}).
 Indeed, the sequence $S_n$ has the same distribution as
 \[
 Y_n = \sqrt{n}\left(\frac{1}{n}\sum_{j=1}^n \left|B_{j+1} - B_{j}\right|^p - c_p\right) \,.
 \]
 and it suffices to consider the case that the fractional noise $X_j = B_{j+1} - B_{j}$ and the function $g(x) = |x|^p - c_p$ that has Hermite rank $2$.  More precisely,   if $N \leq p < N+1$ where $N \geq 2$ is an integer, then the function  $g $ belongs to  $\mathcal{D}^N:= \cap_{q\ge 1}\D^{N,q} (\R,\gamma)$  and Corrollary \ref{cor1} gives the convergence rate  to zero of $ d_{\rm TV}(S_n/\sigma_n, Z)$ with $\alpha =2-2H$. Here are some examples. \\

\noindent
{\it Example 1:} Let $p=2.5$ and $\sigma_n^2 = \EE(S_n^2) = \EE(Y_n^2)$. Then $g \in \mathcal{D}^2$ and
			$$d_{\rm TV}(S_n/\sigma_n, Z)   \leq 
			\begin{cases}
				C n^{-\frac{1}{2}} & {\rm if} \ H \in (0, \frac{1}{2}) \,,\\
				 C n^{-\frac{1}{2}} (\log n)^\frac{3}{2}   & {\rm if} \ H=\frac{1}{2} \,,\\
				C n^{3H-2} & {\rm if} \ H \in (\frac{1}{2}, \frac{2}{3}) \,.
			\end{cases}$$

\noindent
{\it Example 2:} Let $p=3$ and $\sigma_n^2 = \EE(S_n^2) = \EE(Y_n^2)$. Then $g \in \mathcal{D}^3$ and
			$$d_{\rm TV}(S_n/\sigma_n, Z)   \leq 
			\begin{cases}
				C n^{-\frac{1}{2}} & {\rm if} \ H \in (0, \frac{1}{2}) \,,\\
				 C  n^{-\frac{1}{2}}\log n  & {\rm if} \ H=\frac{1}{2} \,,\\
				C n^{2H - \frac{3}{2}} & {\rm if} \ H \in (\frac{1}{2}, \frac{3}{4}) \,.
			\end{cases}$$

\noindent
{\it Example 3:} Let $p=4$ and $\sigma_n^2 = \EE(S_n^2) = \EE(Y_n^2)$. Then $g \in \mathcal{D}^4$ and
			$$d_{\rm TV}(S_n/\sigma_n, Z)   \leq 
			\begin{cases}
				C n^{-\frac{1}{2}} & {\rm if} \ H \in (0, \frac{1}{2}) \,,\\
				 C  n^{-\frac{1}{2}} \sqrt{\log n}  & {\rm if} \ H=\frac{1}{2} \,,\\
				C n^{H-1} & {\rm if} \ H \in (\frac{1}{2}, \frac{2}{3}] \,,\\
				C n^{4H-3} & {\rm if} \ H \in (\frac{2}{3}, \frac{3}{4}) \,.
			\end{cases}$$

\subsection{Application to the estimation of the Hurst parameter}

As an  application of the convergence rates of  power variations, we establish the consistency of the estimatior of the Hurst parameter $H$ for the fBm, defined by means of $p$-power variations. This problem has been studied for $H>\frac{1}{2}$ using quadratic variations in the papers \cite{bcij, il97,km12,tv09} and the references therein. In the paper \cite{co10}, a consistent estimator based on the $p$-power variation is adopted, defined as
$$\tilde H = \frac{\log C_p - \log (n^{-1} V_n^p(B))}{p\log n} \,,$$
where the specific constant $C_p$ depends on $p$. In the paper \cite{co10}, the author also discusses other filters to define the power variation and obtains the convergence rate  $1/\sqrt{n}\log n$.
Here we construct another estimator based on  the $p$-power variation, which is motivated by the papers \cite{bcij,km12}, where the quadratic variation is used. 

Let $\lambda > 1, \lambda \in \mathbb{N}$ be a scaling parameter.  Fix $p\ge 2$, and consider the statistics $T_{\lambda,n}$ defined by 
\[
  T_{\lambda, n} := \frac{V^p_{\lambda n} (B)}{V^{p}_n (B)} = \frac{\sum_{j=0}^{\lambda n-1}\left|B_{\frac{j+1}{\lambda n}} - B_{\frac{j}{\lambda n}}\right|^p}{\sum_{j=0}^{n-1}\left|B_{\frac{j+1}{n}} - B_{\frac{j}{n}}\right|^p} \,.
\]
Then we propose the following estimator for the Hurst parameter $H$:
\begin{equation}\label{H.est}
  \hat H_{\lambda, n} = \frac{1}{p}\left(1 - \frac{\log T_{\lambda, n}}{\log \lambda}\right) \,.
\end{equation}
In the next proposition we show the consistency of this estimator. Though the consistency could be clearly obtained from the ergodic theorem, we will apply the main results obtained in this paper to prove the consistency as well as the convergence rate.
\begin{prop} When $H \in (0, \frac{3}{4})$, for $p \in  \{2\} \cup [3, \infty) $,
$$\lim_{n \to \infty} \sqrt{\frac{n}{\log n}}\left(\hat H_{\lambda, n} - H\right) = 0\,,$$
in probability.
\end{prop}
\begin{proof}
	Denote $\alpha_n = n^{-1+pH} V^{p}_n (B)$. Then $$\log \alpha_{\lambda n} - \log \alpha_n = (-1+pH)\log \lambda +  \log T_{\lambda, n} \,.$$ Thus 
	\begin{equation}\label{h.est}
		\hat H_{\lambda, n} - H = - \frac{\log \alpha_{\lambda n} - \log \alpha_n}{p \log \lambda} \,.
	\end{equation}
    Let $\sigma_n^2 = \EE[(\sqrt{n}(\alpha_n - c_p))^2]$. By previous results, we know that $\sqrt{n}(\alpha_n - c_p) \to \sigma_{H,p} Z$ where $\sigma_n^2 \to   \sigma^2_{H,p}$, and
	  $$d_{\rm TV}(\frac{\sqrt{n}(\alpha_n - c_p)}{\sigma_n}, Z) < n^{-a}$$
	  for some $a>0$. Then for any $\epsilon > 0$,
	    \begin{eqnarray*}
			P\left(\left|\frac{\sqrt{n}(\alpha_n - c_p)}{\sigma_n}\right| > \epsilon \sqrt{\log n}\right) \leq P(|Z| > \epsilon \sqrt{\log n}) + n^{-a} \\
			\leq \frac{C_\epsilon}{n^{\frac{\epsilon^2}{2}}\sqrt{\log n}} + n^{-a},
		\end{eqnarray*}
where we have used the estimate for the tail of a standard Gaussian random variable, i.e., $P(Z > x) \leq \frac{e^{-x^2/2}}{x\sqrt{2\pi}}$.
	This implies that $\frac{\sqrt{n}(\alpha_n - c_p)}{\sqrt{\log n}} \to 0$ in probability as $n \to \infty$. Back to  equation \eqref{h.est}, note that
	   $\log \alpha_n - \log c_p = \frac{1}{\alpha_n^*} (\alpha_n - c_p)$ for some $\alpha_n^*$ between $\alpha_n$ and $c_p$. These results are true for $\alpha_{\lambda n}$ as well, so we conclude that $\sqrt{\frac{n}{\log n}}\left(\hat H_{\lambda, n} - H\right) \to 0$ in probability.
\end{proof}
%============================================

%============================================	
  \section{Appendix}
%============================================
	
  In this section we show some technical lemmas that play a crucial role in the proof of our main results.
	
 \begin{lemma}\label{cov.pth1}  Under the notation and assumptions of Theorem \ref{bm}, let $I_1$ and $I_2$ be the random variables defined in
 (\ref{I1}) and (\ref{I2}), respectively. Suppose $d=2$. Then we have the following estimates.
 \begin{enumerate}
	 \item If $g \in  \D^{3,4} (\R,\gamma)$, then for $i=1,2$, we have
	    $$|\mathbb{E}(I_i)| \leq C \sum_{i \neq 1} |\rho(l_1 - l_i)| \,.$$
	 \item If $g \in \D^{4,4} (\R,\gamma)$, then for $i=1,2$, we have
	\begin{equation}
		|\mathbb{E}(I_i)| \leq C |(\rho(l_1 - l_2) + \rho(l_1 - l_4)) \sum_{j \neq 3} \rho (l_j - l_3) +  \rho(l_1 - l_3)| \,. \label{cov.pth1.eq1}
	\end{equation}
     \item If $g$ is the Hermite polynomial $x^2-1$, then
    $$|\mathbb{E}(I_i)| \leq C |\rho(l_1 - l_3)| \,.$$
  \end{enumerate}
\end{lemma}

\begin{proof}
 We first consider the term $I_1$.  Observe that 
 \[
 g_1(X_{l_1})= \delta (g_2(X_{l_1})l_1).
 \]
  Applying the duality relationship (\ref{dua}), we obtain
     \begin{eqnarray*}
		 \mathbb{E}(I_1) &=& \sum_{a+b+c=1}\EE(g^{(a+2)}(X_{l_2}) g^{(b+2)}(X_{l_4}) (g_1)^{(c)} (X_{l_3}) g_2(X_{l_1})) \\
		 & & \ \times \langle l_1,l_2^{\otimes a} \otimes l_4^{\otimes b} \otimes l_3^{\otimes c}  \rangle_{\mathfrak{H}} \,.
	 \end{eqnarray*}
	When $g$ is the Hermite polynomial $x^2-1$, we just need to consider the case $a=0$, $ b=0$ and $c=1$. In this way we  get
	    \[
	    |\mathbb{E}( I_1)| \leq C |\rho(l_1 - l_3)| \,.
	    \]
	When $g \in \D^{3,4} (\R,\gamma)$, we obtain
	    \[
	    |\mathbb{E} (I_1)| \leq C \sum_{i \neq 1}| \rho(l_1 - l_i)| \,.
	    \]
	When $g \in \D^{4,4} (\R, \gamma)$, in the case of $c=0$, we apply duality again to obtain
	 \begin{eqnarray*}
		 \mathbb{E} (I_1) &=& \sum_{a+b=1} \sum_{a'+b'+c'=1} \EE(g^{(a+a'+2)}(X_{l_2}) g^{(b+b'+2)}(X_{l_4})  g_2(X_{l_3})  g_2^{(c')}(X_{l_1})) \\
		 & & \times  \langle l_1, l_2^{\otimes a} \otimes  l_4^{\otimes b} \rangle_{\mathfrak{H}} \langle  l_3, l_2^{\otimes a'} \otimes l_4^{\otimes b'} \otimes l_1^{\otimes c'} \rangle_{\mathfrak{H}} \,.
	\end{eqnarray*}
Then the inequality \eqref{cov.pth1.eq1} for $i=1$ is derived from expanding the above identities.

Similarly, for the term $I_2$, since $g'(X)$ has the Hermite rank $1$, we can write 
     \[
     g'(X_{l_i}) = \delta\left(  (g')_1(X_{l_i}) l_i\right) \,.
     \]
Using this representation, we have
     \[
     \EE(I_2) = \EE\left(\delta\left( (g')_1(X_{l_1}) l_1\right)\delta\left( (g')_1(X_{l_3}) l_3\right)  g_1'(X_{l_2})  g_1'(X_{l_4})\right) \,.
     \]
We use the similar arguments as the term $I_1$ to obtain the inequality \eqref{cov.pth1.eq1} for $i=2$. 
\end{proof}

 \begin{lemma}\label{cov.pth2}  Under the notation and assumptions of Theorem \ref{bm}, let $I_1$ and $I_2$ be the random variables defined in
 (\ref{I1}) and (\ref{I2}), respectively. Suppose $d \ge 3$.
Then for $i=1, 2$,
    \begin{eqnarray*}
        |\EE(I_i)| \leq C \sum_{\beta \in \mathcal{I}_1} |\rho(l_1-l_2)^{\beta_1} \rho(l_1-l_3)^{\beta_2} \rho(l_1-l_4)^{\beta_3}  \rho(l_3-l_2)^{\beta_4} \rho(l_2-l_4)^{\beta_5} \rho(l_3-l_4)^{\beta_6}| \,,
    \end{eqnarray*}
where $\beta=(\beta_1, \ldots, \beta_6)$, $\mathbb{N}_0 = \mathbb{N}\cup \{0\}$ and 
\begin{eqnarray} \nonumber
	\mathcal{I}_1 = \{\beta \in \mathbb{N}_0^6 &: & d-1 \leq \beta_1 + \beta_2 + \beta_3, \ d-1 \leq \beta_2 + \beta_4 + \beta_6 ,  \ d-2 \leq \beta_1 + \beta_4 + \beta_5, \\  \label{equ60}
	& &  d-2 \leq \beta_3 + \beta_5 + \beta_6, \ \sum_{i=1}^6 \beta_i \leq 3d-4 \}.
\end{eqnarray}
Moreover, if $g$ is the Hermite polynomial $H_d$,  we obtain
    \begin{equation}  \label{equ94}
        |\EE(I_i)| \le C \sum_{\beta \in \mathcal{I}_3} |\rho(l_1 - l_2)^{\beta_1} \rho(l_1 - l_3)^{\beta_2} \rho(l_1 - l_4)^{\beta_3}  \rho(l_3 - l_2)^{\beta_3} \rho(l_2 - l_4)^{\beta_2-1} \rho(l_3 - l_4)^{\beta_1}| \,,
    \end{equation}
 where
\[
	\mathcal{I}_3= \{\beta=(\beta_1, \beta_2, \beta_3)  \in \mathbb{N}^3  : \beta_1 + \beta_2 + \beta_3 = d-1 \}.
\]
\end{lemma}

\begin{proof}
We can represent the factor $g_1(X_{l_1})$ appearing in $I_1$ as  $g_1(X_{l_1}) =  \delta ^{d-1} ( g_d(X_{l_1}) l_1^{\otimes (d-1)})$.
Then applying the duality relationship (\ref{dua2})  and Leibniz's rule yields
  \begin{eqnarray*}
  	\EE(I_1) & = & \sum_{a+b+c=d-1} \EE\left(g^{(a+2)}(X_{l_2}) g^{(b+2)}(X_{l_4})  g_d(X_{l_1}) g_1 ^{(c)} (X_{l_3}) \right) \\
	&& \qquad  \times  \rho(l_1 - l_2)^a \rho(l_1 - l_4)^b \rho( l_1-l_3)^c \,.
  \end{eqnarray*}
  We write 
  \[
  g_1 ^{(c)} (X_{l_3}) = \delta^{d-1-c} ( T_{d-1-c} (g^{(c)}_1) (X_{l_3}) l_3 ^{\otimes (d-1-c)}).
  \]
 Then, applying again  the duality relationship (\ref{dua2})  and Leibniz's rule, we obtain
    \begin{eqnarray*}
    	\EE(I_1)    
	   &=& \sum_{a+b+c=d-1} \sum_{a'+b'+c' = d-1-c}
	  \EE \Big(g^{(a+a'+2)}(X_{l_2}) g^{(b+b'+2)}(X_{l_4}) \\
	  && \qquad     \times g_d^{(c')} (X_{l_1})T_{d-1-c}  (g_1^{(c)})(X_{l_3})\Big)  \\
	  &&  \qquad  \times \rho(l_1 - l_2)^a \rho(l_1 - l_4)^b \rho(l_1 - l_3)^{c+c'} \rho(l_3 - l_2)^{a'}\rho(l_3 - l_4)^{b'}.
    \end{eqnarray*}
    We can still represent the factors $g^{(a+a'+2)}(X_{l_2})$ and $g^{(b+b'+2)}(X_{l_4})$ as divergences:
    \[
    g^{(a+a'+2)}(X_{l_2}) = \delta^{d-(a+a'+2)}(T_{d-(a+a'+2)}(g^{(a+a'+2)})(X_{l_2})   l_2 ^{\otimes (d-(a+a'+2))}  )
    \]
    and
       \[
    g^{(b+b'+2)}(X_{l_4}) = \delta^{d-(b+b'+2)}(T_{d-(b+b'+2)}(g^{(b+b'+2)})(X_{l_4})   l_4 ^{\otimes (d-(b+b'+2))}  ).
    \]
    Then,  we repeat the above process to obtain, using the fact that  $g\in \mathcal{D}^{3d-2}$,
\begin{eqnarray}  \nonumber
       |\EE(I_1)|  &\leq&  C \sum  |\rho(l_1 - l_2)^{a+b''} \rho(l_1 - l_4)^{b+b'''} \rho(l_1 - l_3)^{c+c'} \\
       && \qquad  \times  \rho(l_3 - l_2)^{a'+c''}\rho(l_3 - l_4)^{b'+c'''} \rho(l_2 - l_4)^{a''+a'''}|,  \label{equ50} 
   \end{eqnarray}
   where the sum runs over all nonnegative integers $a,b,c, a', b',c', a'', b'', c'', a''', b''', c'''$ satisfying
   \begin{eqnarray*}
   a+b+c &=& d-1 \\
   a' +b'+c' &=& d-1-c \\
   a'' +b''+c'' &=& (d-a'-a-2) \vee 0 \\
   a'''+b'''+c''' &=& (d-b-b' -a'' -2)  \vee 0
   \end{eqnarray*}
    Inequality (\ref{equ50})  can be equivalently written as
    \begin{eqnarray*}
        |\EE(I_1)| \leq C \sum_{\beta \in \mathcal{I}_1} |\rho(l_1 - l_2)^{\beta_1} \rho(l_1 - l_3)^{\beta_2} \rho(l_1 - l_4)^{\beta_3}  \rho(l_3 - l_2)^{\beta_4} \rho(l_2 - l_4)^{\beta_5} \rho(l_3 - l_4)^{\beta_6}| \,,
    \end{eqnarray*}
where $\beta=(\beta_1, \ldots, \beta_6)$ and  $\mathcal{I}_1 $ is the set defined in (\ref{equ60}).
Notice that we have the lower bound $\sum_{i=1}^6 \beta_i \ge 2d-3$. On the other hand, the upper bound $\sum_{i=1}^6 \beta_i   \le 3d-4$ is attained when
$a=d-1$, $a'=d-1$, $a'''=d-2$ and the other numbers vanish. Taking into account that in this case  the function $g''$  might be differentiated $3d-4$ times, we need  $g\in \mathcal{D}^{3d-2}$.
 
When $g$ is the Hermite polynomial  $H_d$, $g_d = 1$ and $g_1 = H_{d-1}$, so we have $T_{d-1-c}  (g_d^{(c)} )= (d-1)(d-2)\cdots (d-c)$. In this case, taking into account of the orthogonality of  Hermite polynomials of different order, we obtain
   \begin{eqnarray*}
   	 |\EE(I_1)| & \leq &  C\sum_{a+b+c=d-1, a'+b'= d-1-c, a+a'= b+b'=\tilde c} |\rho(l_1 - l_2)^a \rho(l_1 - l_4)^b \rho(l_1 - l_3)^{c} \\
     & & \ \times \rho(l_3 - l_2)^{a'}\rho(l_3 - l_4)^{b'} \rho(l_2 - l_4)^{d-2-\tilde c}| \,.
   \end{eqnarray*}
   Again this can be written as
    \begin{eqnarray*}
        |\EE(I_1)| \leq C \sum_{\beta \in \mathcal{I}_2} |\rho(l_1 - l_2)^{\beta_1} \rho(l_1 - l_3)^{\beta_2} \rho(l_1 - l_4)^{\beta_3}  \rho(l_3 - l_2)^{\beta_4} \rho(l_2 - l_4)^{\beta_5} \rho(l_3 - l_4)^{\beta_6}| \,,
    \end{eqnarray*}   
   where $\mathcal{I}_2$ is the set of  $\beta \in \mathbb{N}_0^6$ such that $\beta_1 + \beta_2 +\beta_3=d-1$, $\beta_4+ \beta_6 + \beta_2=d-1$ and
   $\beta_1 + \beta _4= \beta_3 + \beta_6 =d- 2-\beta_5$.  This implies $\beta_1 =\beta_6 $, $\beta_3 =\beta_4$, $\beta_5 = \beta_2 -1$ and $\beta_1+ \beta_2 +\beta_3 =d-1$, and this completes the proof of (\ref{equ94}).
   
   Similar arguments could be applied to  handle the term $I_2$.
\end{proof}

\begin{lemma}\label{a3.c4} Assume condition (\ref{h1}). Define
\[
J_1= \frac 1{n^3} \sum_{l_1, \ldots, l_6=1}^n   \sum_{i=1 \atop i \not =2}^6
	 \sum_{s\in\{3,5,6\}  \atop s\not =i} \sum_{j=1 \atop j\not =s }^6  |\rho_{2i} \rho_{sj}  \rho_{12}\rho_{13}\rho_{45}\rho_{46}|
	 \]
	 and
	  \begin{eqnarray*}	 
	  J_2 &:=&   \frac 1{n^3}  \sum_{l_1, \ldots, l_6=1}^n  \Big(\sum_{\substack{i \neq s \neq j \\ i,s,j \in \{3,5,6\}}}|\rho_{2i} \rho_{sj} \rho_{12}\rho_{13}\rho_{45}\rho_{46}|   \\
		&& +   
		\sum_{ (i,s,j,t,h)  \in D_3}\ |\rho_{2i} \rho_{sj} \rho_{th} \rho_{12}\rho_{13}\rho_{45}\rho_{46}|
		  \Big),\\
			\end{eqnarray*}
	where  the set $D_3$ has been defined in (\ref{d3}) and we recall that  $\rho_{ij} = \rho(l_i - l_j)$.	Then,
		\begin{equation} \label{equ9}
		J_1 \le  \frac Cn
  \left(\sum_{|k| \leq n} |\rho(k)|\right)^2
  \end{equation}
  and
	 \begin{equation} \label{equ10}
	 J_2 \le  \frac C n \sum_{|k| \leq n}|\rho(k)| +  \frac Cn \left(\sum_{|k| \leq n} |\rho(k)|^{\frac{3}{2}}\right)^4 \,.
	 \end{equation}
\end{lemma}

\begin{proof}
{\it Step 1: } We show  first  the inequality (\ref{equ9}).   We make change of variables  $l_1 - l_2 = k_1$, $l_1 - l_3 = k_2$, $l_4 - l_5 = k_3$, $l_4 - l_6=k_4$.   We first consider the term $\rho_{2i}$ that has three possibilities: $\rho(k_1)$, $\rho(k_1 - k_2)$, or a new factor $\rho(k_5)$ where $k_5 = l_2 - l_i$ is linearly  independent of $k_t, t=1, \ldots, 4$. If $\rho_{2i}$ is one of the first two cases, $\rho_{sj}$ have three possibilities: $\rho(k_i)$ for $i=2,3,4$; $\rho(k_1-k_2)$ or $\rho(k_3 - k_4)$; a new factor $\rho(k_5)$ where $k_5 = l_j - l_s$ independent of $k_t$, $1\le t \le 4$.  If $\rho_{2i}$ is in the third case, i.e. a new factor, then $\rho_{sj}$ have several possibilities: $\rho(k_i)$ for $i=2,3,4$; $\rho({\bf k} \cdot {\bf v})$ where ${\bf k} \cdot {\bf v}$ is a linear combination of two, three or four or five $k_t$'s, $1\le  t \le 5$. Through this analysis, by taking advantage of the symmetry, we obtain
	    \[
	    J_1 \leq  \frac C{n^2}  \sum_{i=1}^9       \sum_{|k_j| \leq n, 1\le j \le 5} |J_{1i}|,
	    \]
where	
\begin{eqnarray*}
		&& J_{11} = \rho(k_1)^2 \rho(k_2)^2 \rho(k_3) \rho(k_4), \\
		&& J_{12} = \rho(k_1)^2 \rho(k_2) \rho(k_1 - k_2)\rho(k_3) \rho(k_4), \\
		&& J_{13} = \rho(k_1)^2 \rho(k_2) \rho(k_3) \rho(k_4) \rho(k_3 - k_4), \\
		&& J_{14} = \rho(k_1)^2 \rho(k_2) \rho(k_3) \rho(k_4) \rho(k_5),\\
		&& J_{15} = \rho(k_1) \rho(k_2) \rho(k_1-k_2) \rho(k_3) \rho(k_4) \rho(k_3-k_4), \\
		&& J_{16} = \rho(k_1) \rho(k_2) \rho(k_1-k_2) \rho(k_3) \rho(k_4) \rho(k_5),\\
		&& J_{17} = \rho(k_1) \rho(k_2) \rho(k_3) \rho(k_4) \rho(k_5)\rho(k_1 - k_5 - k_2) ,\\
		&& J_{18} = \rho(k_1) \rho(k_2) \rho(k_3) \rho(k_4) \rho(k_5) \rho(k_1 - k_2 + k_3 - k_4) ,\\
		&& J_{19} = \rho(k_1) \rho(k_2) \rho(k_3) \rho(k_4) \rho(k_5) \rho(k_1 - k_2 + k_3 - k_4 + k_5).
	\end{eqnarray*}
	We claim that for $i=1,\dots, 9$, the following estimate holds true
	\begin{equation} \label{equ11}
  \frac 1{n^2}   \sum_{|k_j| \leq n, 1\le j\le 5}  |J_{1i}| \leq \frac{C}{n} \left(\sum_{|k| \leq n} |\rho(k)|\right)^2.
  \end{equation}
  The estimate  (\ref{equ11}) holds clearly for $i=1$ and $i=4$ due to condition (\ref{h1}).
By the Cauchy-Schwartz inequality   we have
\[
 \sum_{|k_1|, |k_2| \leq n}  \rho(k_1)^2 |\rho(k_2) \rho(k_1 - k_2)| <\infty
 \]
 and (\ref{equ11}) is true for $i=2$. For $i=3, 5, 6$, the estimate (\ref{equ11}) follows from  (\ref{equ6}) and   (\ref{equ7}) with $M=2$ and
 for $i=7,8,9$ we use these inequalities with $M=3,4,5$, respectively.

   \medskip
   \noindent
   		{\it Step 2:} We proceed to prove the inequality (\ref{equ10}). Note that  for the first summand in $J_2$, the product $\rho_{2i} \rho_{sj}$ can be only one of the following terms:  $\rho_{23} \rho_{56}$,   $\rho_{26} \rho_{35}$, or  $\rho_{25} \rho_{36}$. In the first case, we obtain the term $J_{15}$, for which we have, by (\ref{equ6}) with $M=2$, 
		\[
		 \frac 1{n^2}    \sum_{|k_j| \leq n, 1 \le j\le 5}  |J_{15}|   \le  \frac Cn \left(\sum_{|k| \leq n} |\rho(k)|^{\frac{3}{2}}\right)^4.
		 \]
		In the second and third case, we obtain the term  $J_{19}$, for which we have, by (\ref{equ6}) with $M=5$, 
		\begin{equation} \label{equ13}
		 \frac 1{n^2}   \sum_{|k_j| \leq n, 1 \le j\le 5}  |J_{19}|   \le  \frac C{n^2}  \left(\sum_{|k| \leq n} |\rho(k)|^{\frac{6}{5}}\right)^5.
		 \end{equation}
		By H\"older's inequality,
		\begin{equation} \label{equ14}
		\left(\sum_{|k| \leq n} |\rho(k)|^{\frac{6}{5}}\right)^5 \le  n\left(\sum_{|k| \leq n} |\rho(k)|^{\frac{3}{2}}\right)^4,
		\end{equation}
		and we obtain the desired bound.
		
		Let us now consider the second summand in the expression of $J_2$.
	 This summand will consists of terms of the form  $ J_{1i} \rho_{th}$ for $i=1,\ldots, 4, 6, \ldots, 9$, where $\rho_{th}$ can be written as a linear combination of $k_1, \ldots, k_5$.  For $i=6, \dots, 8$, we  estimate the factor $|\rho_{th}|$ by one and
	 apply the estimate (\ref{equ6}) with $M=3,4,5$ to obtain
	 \begin{equation}  \label{equ41}
	  \frac 1{n^2}     \sum_{|k_j| \leq n, 1\le j\le 5}  |J_{16}| \le  \frac C {n^2}  \left(\sum_{ |k| \le n} |\rho(k)| \right)^3  \left(\sum_{ |k| \le n} |\rho(k)| ^{\frac 32}\right)^2,
	  \end{equation}
	    \begin{equation}  \label{equ42}
	  \frac 1{n^2}   \sum_{|k_j| \leq n, 1\le j\le 5}  |J_{17}| \le  \frac C {n^2}  \left(\sum_{ |k| \le n} |\rho(k)| \right)^2  \left(\sum_{ |k| \le n} |\rho(k)| ^{\frac 43}\right)^3,
	   \end{equation}
	  and
       \begin{equation}  \label{equ43}
	  \frac 1{n^2}   \sum_{|k_j| \leq n, 1\le j\le 5}  |J_{18}| \le  \frac C {n^2}  \left(\sum_{ |k| \le n} |\rho(k)| \right)  \left(\sum_{ |k| \le n} |\rho(k)| ^{\frac 54}\right)^4.
	  	   \end{equation}
	  Then, from (\ref{equ41}) and \eqref{ho2}, we get
	  \[
	  \frac 1{n^2}     \sum_{|k_j| \leq n, 1\le j\le 5}  |J_{16}| \le  \frac C {n}\left(\sum_{ |k| \le n} |\rho(k)| ^{\frac 32}\right)^4.
	  \] 
	  From   (\ref{equ42}), (\ref{ho1}) with $M=3$ and (\ref{ho2})
	  \[
	   \frac 1{n^2}   \sum_{|k_j| \leq n, 1\le j\le 5}  |J_{17}| \le
	   \frac C {n}\left(\sum_{ |k| \le n} |\rho(k)| ^{\frac 32}\right)^4.
	   \]
	   Finally, from (\ref{equ43}), (\ref{ho1}) with $M=4$ and the above inequality of $J_{17}$,
	  \[
	   \frac 1{n^2}   \sum_{|k_j| \leq n, 1\le j\le 5}  |J_{18}| \le
	   \frac C {n}\left(\sum_{ |k| \le n} |\rho(k)| ^{\frac 32}\right)^4.
	   \]
	  The term  $J_{19}$ can be handled  applying
	   (\ref{equ13}) and (\ref{equ14}).
	  
	    For $J_{11}, J_{12}$, $t$ can be just chosen from  the set $\{5, 6\}$ and the possible values of the factor $\rho_{th}$ (after a change of variable) can be $\rho(k_3), \rho(k_4), \rho(k_3 - k_4)$ or $\rho(k_5)$ where $k_5$ is linearly independent of $k_1, \ldots, k_4$. Then we first sum up the variables $k_1$ and $k_2$ and this part produces a constant. The sum with respect to $k_3, k_4, k_5$ is as follows.
 \[
 \sum_{|k_j| \leq n} |\rho(k_3)^2 \rho(k_4)| \leq C \sum_{|k| \leq n} |\rho(k)|\,, \ \sum_{|k_j| \leq n} |\rho(k_3) \rho(k_4) \rho(k_5)| = (\sum_{|k| \leq n} |\rho(k)|)^3 \leq n \sum_{|k| \leq n} |\rho(k)|
 \]
and 
\[
\sum_{|k_j| \leq n} |\rho(k_3) \rho(k_4) \rho(k_3 - k_4)| \leq C \sum_{|k| \leq n} |\rho(k)|\,,
\]
where we have used \eqref{equ6} and (\ref{equ7}) with $M=2$. Therefore,
     \[
    \frac 1{n^2} \sum_{j=1}^5  \sum_{|k_j| \leq n} |J_{1i} \rho_{th}| \leq \frac{C}{n} \sum_{|k| \leq n} |\rho(k)|  \,, \ i=1,2 \,.
    \]
For $J_{13}$, $t=3$  and  possible values of $\rho_{th}$ can be $\rho(k_2), \rho(k_2 - k_1)$ or $\rho(k_5)$ where $k_5$ is linearly independent of $k_1, \ldots, k_4$. The first two cases have been considered above in the discussion of the terms $J_{11} \rho_{th}$ and $J_{12} \rho_{th}$. For the third case, observe that 
     \begin{eqnarray*}
		    \frac 1{n^2} \sum_{j=1}^5  \sum_{|k_j| \leq n} |\rho(k_1)^2 \rho(k_2) \rho(k_3) \rho(k_4) \rho(k_3 - k_4) \rho(k_5)| \\
		 \leq  \frac  C{n^2}   \left(\sum_{|k| \leq n} |\rho(k)| \right)^3    \leq  \frac Cn \sum_{|k| \leq n} |\rho(k)|.
	\end{eqnarray*}
where we have used \eqref{equ6} and \eqref{equ7} with $M=2$. Thus,
     \[
       \frac 1{n^2} \sum_{j=1}^5  \sum_{|k_j| \leq n}  |J_{13} \rho_{th}| \leq \frac{C}{n} \sum_{|k| \leq n} |\rho(k)| .
       \]
Finally, for $J_{14}$, the term $\rho_{th}$ could be $\rho(k_i), i=2,\ldots, 4$ or $\rho(\star)$ where $\star$ is a linear combination of $k_i$'s which at least involves  two different terms $k_{h_1}$ and  $k_{h_2}$ where $h_1, h_2 \in \{2,3,4,5\}$. The first case has been considered above in  the discussion of the terms $J_{1i} \rho_{th}, i = 1, 2,3$. For the second case, we apply
inequalities \eqref{equ6} and \eqref{equ7} with $M=2,3,4$ and we get
   \begin{eqnarray*}
	  \frac 1{n^2} \sum_{j=1}^5  \sum_{|k_j| \leq n}  |\rho(k_1)^2 \rho(k_2) \rho(k_3) \rho(k_4) \rho(k_5) \rho(\star)| \\
	 \leq   \frac C {n^2}  \left(\sum_{|k| \leq n} |\rho(k)| \right)^3    \leq  \frac Cn \sum_{|k| \leq n} |\rho(k)|.
\end{eqnarray*}
Thefore, $  \frac 1{n^2} \sum_{j=1}^5  \sum_{|k_j| \leq n} |J_{14} \rho_{th}| \leq \frac{C}{n} \sum_{|k| \leq n} |\rho(k)|$ and this finishes the proof.
\end{proof}

\begin{lemma}\label{a4.c6}
  Define 
	\[
	\mathcal{L}_1:= n^{-4} \sum_{l_1, \ldots, l_8=1}^n  \sum_{ (i,s,j,t,h) \in D_4} |\rho_{12} \rho_{13} \rho_{14} \rho_{56} \rho_{57} \rho_{58}\rho_{2i} \rho_{sj} \rho_{th} |.
	\]
	where the set $D_4$  has been defined in (\ref{d4}). Then
		\begin{equation} \label{equ20}
	\mathcal{L}_1 \le \frac{C}{n^2} \left(\sum_{|k| \leq n} |\rho(k)|\right)^{3} .
	\end{equation}
	\end{lemma}

\begin{proof}
     We make  the change of variables  $l_1-l_2= k_1$, $l_1-l_3 =k_2$, $l_1-l_4 =k_3$,  $l_5-l_6= k_4$, $l_5-l_7 =k_5$, $l_5-l_8 =k_6$.
  The  factors $ \rho_{2i} $, $\rho_{sj} $ and $\rho_{th}$ can be  of one of the two forms:
  \begin{itemize}
  \item[(i)] $\rho_{\alpha \beta}$,  where $\alpha, \beta \in \{1,2,3,4\}$  or  $ \alpha, \beta \in \{5,6,7,8\}$.
  \item[(ii)] $\rho_{\alpha \beta}$,  where  $\alpha \in \{1,2,3,4\}$ and  $\beta \in \{5,6,7,8\}$ or  $\beta \in \{1,2,3,4\}$ and  $\alpha \in \{5,6,7,8\}$.
  \end{itemize}
  For factors of the form (i), we have $\rho_{\alpha \beta} = \rho( {\bf k} \cdot  {\bf v})$, where ${\bf k}$ is one of the vectors $(k_1,k_2,k_3) $ or $(k_4,k_5,k_6) $ and ${\bf v}$ is a vector in $\R^4$ whose components are  $0$, $1$ or $-1$. For the first factor of the form (ii), we write $\rho_{\alpha \beta} = \rho(k_7)$, where $k_7$ is a new variable independent of  the $k_i$'s, $1\le i \le 6$. If there are more than one factor of the form (ii), then these extra factor(s) can be written as   $\rho( {\bf k} \cdot  {\bf v})$, where
  ${\bf k}=(k_1,k_2,k_3,k_4,k_5,k_6,k_7) $ and ${\bf v}$ is a vector in $\R^7$ whose components are  $0$, $1$ or $-1$.
  
  Then we decompose  $\mathcal{L}_1$  as the sum of several terms $\mathcal{L}_{1j}$, according to the following cases:\\
  
  \noindent
	 {\it Case 1:} There are three factors that have power $2$. We denote the corresponding term by $\mathcal{L}_{11}$. For this term we have
	 \[
	 \mathcal{L}_{11} = \frac 1 {n^2}  \sum_{\substack{|k_i| \leq n \\ i=1,\ldots, 6}}  \rho(k_1)^2 \rho(k_2)^2\rho(k_3)^2| \rho(k_4) \rho(k_5) \rho(k_6)| \le  \frac{C}{n^2} \left(\sum_{|k| \leq n} |\rho(k)|\right)^{3}.
	 \]
	 \\
	  {\it Case 2:} Two factors have power $2$. Then we have the following possibilities by taking into account of the symmetry.
\[
      	\mathcal{L}_{12}:=  \frac 1 {n^3}  \sum_{\substack{|k_i| \leq n \\ i=1,\ldots, 7}} |\rho^2(k_1) \rho^2(k_2) \rho(k_3) \rho(k_4) \rho(k_5) \rho(k_6)  \rho(k_7) | 
	\]
	and
	\[
		\mathcal{L}_{13}:= \frac 1 {n^2} \sum_{\substack{|k_i| \leq n \\ i=1,\ldots, 6}} |\rho^2(k_1) \rho^2(k_2) \rho(k_3) \rho(k_4) \rho(k_5) \rho(k_6)\rho({\bf k} \cdot  {\bf v}) | ,
		\]
	 where  ${\bf k} =(k_1, k_2, k_3, k_4, k_5, k_6)$ and  ${\bf v}$ is a vector in $\R^6$ whose components are  $0$, $1$ or $-1$. Clearly,
	   \begin{eqnarray*}
	   	  \mathcal{L}_{12} \leq  \frac C  {n^3} \left(\sum_{|k| \leq n} |\rho(k)|\right)^5 \leq  \frac C {n^2} \left(\sum_{|k| \leq n} |\rho(k)|\right)^{3} \,.
	   \end{eqnarray*}
	For $\mathcal{L}_{13}$, $ {\bf k} \cdot {\bf v}$ involves at least two factors $k_{j}, k_{j'}$ but  $ {\bf k} \cdot {\bf v}$ cannot be a linear combination of only $k_1 $ and $k_2$.  Applying inequality (\ref{equ22})  with $M=5$, yields
   \begin{eqnarray*}
   	  \mathcal{L}_{13} \leq n^{-2} \left(\sum_{|k| \leq n} |\rho(k)|\right)^3 \,.
   \end{eqnarray*}
   
   \medskip
   \noindent
    {\it Case 3:} Only one factor has power $2$. Then we have the following two possibilities, taking into account the symmetry. The first one is
 \[
      	\mathcal{L}_{14} = \frac 1 {n^3} \sum_{\substack{|k_i| \leq n \\ i=1,\ldots, 7}}|\rho^2(k_1) \rho(k_2) \rho(k_3) \rho(k_4) \rho(k_5) \rho(k_6) \rho(k_7) \rho({\bf k} \cdot {\bf v})| \,,
	\]
	where   ${\bf k} =(k_1, k_2, k_3, k_4, k_5, k_6, k_7)$ and  ${\bf v}$ is a vector in $\R^7$ whose components are  $0$, $1$ or $-1$
	and it has at least two nonzero  components. By (\ref{equ22}) with $M=7$, we can write
\[
   	  \mathcal{L}_{14}  \leq  \frac C { n^3}  \left(\sum_{|k| \leq n} |\rho(k)|\right)^5 \leq   \frac C { n^2} \left(\sum_{|k| \leq n} |\rho(k)|\right)^{3} \,.
\]
	The second  possibility  is
	\[
		\mathcal{L}_{15}:= \frac 1{ n^2}  \sum_{\substack{|k_i| \leq n \\ i=1,\ldots, 6 }}|\rho^2(k_1) \rho(k_2) \rho(k_3) \rho(k_4) \rho(k_5) \rho(k_6) \rho({\bf k} \cdot {\bf v} ) \rho({\bf k} \cdot {\bf w})| \,,
\]
where  ${\bf k} =(k_1, k_2, k_3, k_4, k_5, k_6)$ and ${\bf v}$, ${\bf w} $ are vectors in $\R^6$ in such a way that
${\bf k} \cdot {\bf v} $ and ${\bf k} \cdot {\bf w}$  are linear combinations of $k_1, k_2, k_3$ or $k_4, k_5, k_6$
with exactly two nonzero components  are equal to $1$ and $-1$ and satisfying some additional restrictions, due to the definition of the set $D_4$. 
 There are several combinations:
 \begin{itemize}
 \item[(i)] ${\bf k} \cdot {\bf v}= k_1-k_2$ and ${\bf k} \cdot {\bf w}$ is either $k_2-k_3$ or $k_1-k_3$. In this case, by H\"older's inequality, we have
 \[
  \sum_{ | k_i | \le n, 1\le i \le 3} |\rho^2(k_1) \rho(k_2) \rho(k_3) \rho({\bf k} \cdot {\bf v} )  \rho({\bf k} \cdot {\bf w})| \le C,
  \]
  and  we obtain
  \begin{equation}  \label{equ30}
   	  \mathcal{L}_{15}   \leq   \frac C { n^2} \left(\sum_{|k| \leq n} |\rho(k)|\right)^{3} \,.
\end{equation}
  \item[(ii)]  ${\bf k} \cdot {\bf v}$ and ${\bf k} \cdot {\bf w}$ are two  different  linear combinations  chosen among
  $\{ k_4-k_5, k_4-k_6, k_5-k_6\}$.  Then, the inequality (\ref{equ23}) with $M=3$ yields
   \[
  \sum_{ | k_i | \le n,  i=4,5,6} |\rho(k_4) \rho(k_5) \rho(k_6) \rho({\bf k} \cdot {\bf v} )  \rho({\bf k} \cdot {\bf w})| \le 
  \sum_{|k| \leq n} |\rho(k)|,
  \]
  which implies  (\ref{equ30}).
  \item[(iii)] If ${\bf k} \cdot {\bf v}= k_1-k_2$ and  ${\bf k} \cdot {\bf w}$ is  $ k_4-k_5$,  $ k_4-k_6$ or $ k_5-k_6$, then  (\ref{equ30})  follows from
  \[
    \sum_{ | k_1 | \le n,  |k_2| \le n} |\rho^2(k_1) \rho(k_2) \rho(k_1- k_2) |   \le C
  \]
  and (\ref{equ21}) with $M=3$.
  \end{itemize}
 
  \medskip
  \noindent 
	 {\it Case 4:} All factors have power $1$, $i \in \{3, 4\}$, and $\rho_{sj}=\rho({\bf k} \cdot {\bf v}), \rho_{th}=\rho({\bf k} \cdot {\bf w})$ where ${\bf k} \cdot {\bf v}$ is a linear combination of $k_1,k_2, k_3$ and  ${\bf k} \cdot {\bf w}$ is a linear combination of $k_4, k_5, k_6$, or vice versa.  We denote the corresponding term by $\mathcal{L}_{16}$. Then the estimate
	 \[
	 \mathcal{L}_{16} \leq n^{-2} \left(\sum_{|k| \leq n} |\rho(k)|\right)^3 \,
	 \]
	 follows from (\ref{equ23}) with $M=3$ and  (\ref{equ21}) with $M=3$. 
	 
	   \medskip
  \noindent 	    
	 {\it Case 5:} All factors have power $1$, and there is one of the  differences $l_i-l_2$, $l_j- l_s$ or $l_h-l_t$ linearly independent of $k_1, \ldots, k_6$. We denote this difference by $k_7$. The other two factors are of the form $\rho({\bf k} \cdot {\bf v})$ and $\rho({\bf k} \cdot {\bf w})$, where ${\bf k} \cdot {\bf v}$ and  ${\bf k} \cdot {\bf w}$ are linear combinations of $k_1, \ldots, k_6, k_7$. 
	 In this case, the desired estimate follows from the inequality (\ref{equ23}), with $M=7$. In fact, if we denote the corresponding term by $\mathcal{L}_{17}$, we obtain
	 \[
   	  \mathcal{L}_{17}  \leq  \frac C { n^3}  \left(\sum_{|k| \leq n} |\rho(k)|\right)^5 \leq   \frac C { n^2} \left(\sum_{|k| \leq n} |\rho(k)|\right)^{3} \,.
\]
	   This finishes the  lemma.
	   \end{proof}
	   
\begin{lemma} \label{a4.c6-2}
Define	   
 	\[
	\mathcal{L}_2:= n^{-4} \sum_{l_1, \ldots, l_8=1}^n    \sum_{\substack{i \neq s \neq j\\ i, s, j\in\{4,7,8\}}}  |\rho_{12} \rho_{13} \rho_{24} \rho_{56} \rho_{57} \rho_{68}\rho_{3i} \rho_{sj}|  
	\]
	and
	\[
	\mathcal{L}_3:= n^{-4} \sum_{l_1, \ldots, l_8=1}^n    \sum_{ (i,s,j,t,h)\in D_5}  |\rho_{12} \rho_{13} \rho_{24} \rho_{56} \rho_{57} \rho_{68}\rho_{3i} \rho_{sj} \rho_{th}|,
\]
	where the set $D_5$ has been defined in  (\ref{d5}).
	Then
	\begin{equation}
	\mathcal{L}_2   \leq  \frac{C}{n^2} \left(\sum_{|k|\leq n}|\rho(k)|^{\frac{4}{3}}\right)^6. \label{ineq.tL}
		\end{equation}
	and
	\begin{equation}
	 \mathcal{L}_3  \leq \frac{C}{n^2} \left(\sum_{|k| \leq n} |\rho(k)|\right)^3. \label{ineq.tL2}
		\end{equation}
\end{lemma}

\begin{proof}
Let us first show (\ref{ineq.tL}). 
We make  the change of variables  $l_1-l_2= k_1$, $l_1-l_3 =k_2$, $l_2-l_4 =k_3$,  $l_5-l_6= k_4$, $l_5-l_7 =k_5$, $l_6-l_8 =k_6$.
 By symmetry, it suffices to analyze the cases $i=4$ and $i=7$. If $i=4$, then $\rho_{34}= \rho(k_1-k_2+k_3)$ and  $s=8, j=7$ or $s=7,j=8$, which gives $\rho_{sj}= \rho(k_4-k_5+k_6)$. In this case, we obtain a term of the form
 \[
  \mathcal{L}_{21}:= n^{-2} \sum_{|k_i| \leq n, i=1,\ldots, 6} |\rho(k_1)\rho(k_2)\rho(k_3)\rho(k_1-k_2+k_3) \rho(k_4)\rho(k_5)\rho(k_6)\rho( k_4-k_5+k_6)|.
  \]
  Applying inequality (\ref{equ6}) with $M=3$ yields
   \[
  \mathcal{L}_{21} \le  \frac{C}{n^2} \left(\sum_{|k|\leq n}|\rho(k)|^{\frac{4}{3}}\right)^6.
  \]
  In the case $i=7$, we set $\rho_{37} = \rho(k_7)$ and  have two possibilities for $sj$:  
  %$47$, $87$
   $48$ and $84$, which produce the following term
\begin{eqnarray*}
	% \mathcal{L}_{22}: &=&  n^{-3} \sum_{|k_i| \leq n \atop i=1,\ldots, 7} |\rho(k_1)\rho(k_2)\rho(k_3)\rho(k_4)\rho(k_5)\rho(k_6)\rho(k_7)   \rho(k_2 + k_7 -k_3 - k_1 )| \\
	  \mathcal{L}_{23}: &=&  n^{-3} \sum_{|k_i| \leq n \atop i=1,\ldots, 7} |\rho(k_1)\rho(k_2)\rho(k_3)\rho(k_4)\rho(k_5)\rho(k_6)\rho(k_7)   \\
	  && \times  \rho(k_2 + k_7 -k_3 - k_1-k_5+k_4+k_6 )|
	%   \mathcal{L}_{24}: &=&  n^{-3} \sum_{|k_i| \leq n \atop i=1,\ldots, 7} |\rho(k_1)\rho(k_2)\rho(k_3)\rho(k_4)\rho(k_5)\rho(k_6)\rho(k_7)   \rho(k_4 - k_5 +k_6 )|.
  \end{eqnarray*}
  Applying the inequality (\ref{equ6}) with  %$M=4$, 
  $M=7$
  % and $M=3$ and using  (\ref{ho1}) with $M=4$ 
  and H\"older's inequality, we obtain
  \begin{eqnarray*}
  %\mathcal{L}_{22} &\le &  \frac C {n^3}  \left(\sum_{|k|\leq n}|\rho(k)|^{\frac{5}{4}}\right)^4 \left( \sum_{|k| \le n} | \rho(k)| \right)^3
  %\le  \frac C {n^3} \left(\sum_{|k|\leq n}|\rho(k)|^{\frac{4}{3}}\right)^3 \left( \sum_{|k| \le n} | \rho(k)| \right)^4,  \\
    \mathcal{L}_{23} &\le &  \frac C {n^3}  \left(\sum_{|k|\leq n}|\rho(k)|^{\frac{8}{7}}\right)^7  
    \le  \frac C {n^2}  \left(\sum_{|k|\leq n}|\rho(k)|^{\frac{4}{3}}\right)^6. \\
  %    \mathcal{L}_{24} &\le &  \frac C {n^3}  \left(\sum_{|k|\leq n}|\rho(k)|^{\frac{4}{3}}\right)^3 \left( \sum_{|k| \le n} | \rho(k)| \right)^4.
  \end{eqnarray*}
  % Using the inequality
  % \[
  %   \left( \sum_{|k| \le n} | \rho(k)| \right)^4 \le n  \left(\sum_{|k|\leq n}|\rho(k)|^{\frac{4}{3}}\right)^3,
  % \]
  %  we obtain the desired result. 
  This finishes the proof of (\ref{ineq.tL}). The proof of  (\ref{ineq.tL2}) is analogous to that of (\ref{equ20}). Namely, we can make the change of variables $l_1-l_2= k_1$, $l_1-l_3 =k_2$, $l_2-l_4 =k_3$,  $l_5-l_6= k_4$, $l_5-l_7 =k_5$, $l_6-l_8 =k_6$, and follow the arguments of (\ref{equ20}). A subtle difference might be the verification of \eqref{equ30}. That is,  the estimation of
		\[
			\mathcal{L}_{15}:= \frac 1{ n^2}  \sum_{\substack{|k_i| \leq n \\ i=1,\ldots, 6 }}|\rho^2(k_1) \rho(k_2) \rho(k_3) \rho(k_4) \rho(k_5) \rho(k_6) \rho({\bf k} \cdot {\bf v} ) \rho({\bf k} \cdot {\bf w})| \,,
	\]
where ${\bf k} \cdot {\bf v}$, ${\bf k} \cdot {\bf w}$ have the following two cases:
   \begin{itemize}
   	\item [(i)] They are linear combinations of $k_4, k_5, k_6$.
	\item [(ii)]${\bf k} \cdot {\bf v}$ is a linear combination of $k_1, k_2, k_3$ ($k_1 - k_2$ with respect to $i=2$ or $k_2-k_1-k_3$ with respect to $i=4$), and ${\bf k} \cdot {\bf w}$ is a linear combination of $k_4, k_5, k_6$.
   \end{itemize}
In the case (i), we apply the inequality \eqref{equ23} with $M=3$ to obtain 
   \begin{equation}\label{equ31}
	   \mathcal{L}_{15} \leq \frac{C}{n^2} \left(\sum_{|k| \leq n} |\rho(k)|\right)^3\,.
   \end{equation}
In the case (ii), we apply  \eqref{equ22} with $M=3$
and  \eqref{equ21} with $M=3$ to obtain the desired the inequality \eqref{equ31}.
 \end{proof}

The next lemma contains several inequalities that are used along the paper.

\begin{lemma}
 Fix an integer $M\ge 2$. We have
\begin{equation}  \label{equ6}
  \sum_{ |k_j| \le n  \atop 1\le j \le M} |\rho( {\bf k}  \cdot  {\bf v} )|   \prod _{j=1} ^M | \rho(k_j) |   \le C \left(\sum_{|k| \leq n} |\rho(k)|^{1+ \frac 1M}\right)^M.
\end{equation}
where ${\bf k} = (k_1, \dots, k_M)$ and ${\bf v} \in \R^M$ is a fixed vector whose  components are $1$or $-1$. Furthermore, if   $\sum_{k \in \mathbb{Z}}
\rho(k)^2<\infty$, then
\begin{equation}  \label{equ7}
\left(\sum_{|k| \leq n} |\rho(k)|^{1+ \frac 1M}\right)^M\le C\left( \sum_{|k| \leq n} |\rho(k)|\right)^{M-1}
\end{equation}
and if  ${\bf v} \in \R^M$ is a nonzero vector whose  components are  $0$, $1$or $-1$
\begin{equation}  \label{equ21}
  \sum_{ |k_j| \le n  \atop 1\le j \le M} |\rho( {\bf k}  \cdot  {\bf v} )|   \prod _{j=1} ^M | \rho(k_j) |   \le C\left( \sum_{|k| \leq n} |\rho(k)|\right)^{M-1}.
\end{equation}
\end{lemma}

\begin{proof}
     Applying the  Brascamp-Lieb inequality   (\ref{BL}), we have
      \[
   \sum_{ |k_j| \le n  \atop 1\le j \le M}  \prod _{j=1} ^M | \rho(k_j) |  |\rho( {\bf k}  \cdot  {\bf v} )| \leq C 
       \prod_{i=1} ^{M+1}  \left(\sum_{|k| \leq n} |\rho(k)|^{\frac{1}{p_i}}\right)^{p_i} \,,
      \]
where   $p_i \le 1$ and  $\sum_{i=1}^{M+1} p_i = M$.  Choosing $p_i =  M/ (M+1) $ for $i=1,\dots, M+1$,  we get   inequality (\ref{equ6}).
To show  (\ref{equ7}), we make the decomposition $  |\rho(k)|^{1+ \frac 1M}= |\rho(k)|^{1- \frac 1M}  |\rho(k)|^{ \frac 2M}$ and apply H\"older's inequality with exponents $p= \frac M  {M-1}$ and $q=M$. Finally, to show (\ref{equ21}), we decompose the sum into the product of the sum with respect to the $k_i$'s that appear in ${\bf k } \cdot {\bf v}$ and the sum of the remaining terms. 
	 \end{proof}

 \begin{lemma}
 Fix an integer $M\ge 3$ and assume  $\sum_{k \in \mathbb{Z}}
\rho(k)^2<\infty$.  We have
\begin{equation}  \label{equ22}
  \sum_{|k_j| \le n  \atop 1\le j \le M}   \rho(k_1)^2 |\rho( {\bf k}  \cdot  {\bf v} )|  \prod _{j=2} ^{M} | \rho(k_j) |   \le C \left(\sum_{|k| \leq n} |\rho(k)|\right)^{M-2},
\end{equation}
where ${\bf k} = (k_1, \dots, k_M)$ and ${\bf v} \in \R^M$ is a fixed vector whose  components are  $0
$, $1$or $-1$ and it has at least two nonzero components.   
\end{lemma}

\begin{proof}
It suffices to assume that all the components of ${\bf v}$ are nonzero. In this case, we apply  the  Brascamp-Lieb inequality    (\ref{BL}) with exponents $p_1=1$ and $p_2 = \cdots =p_{M+1} = \frac  {M-1}M$ and inequality (\ref{equ7})  with $M$ replaced by $M-1$. 
\end{proof}

\begin{lemma}
 Fix an integer $M\ge 3$ and assume  $\sum_{k \in \mathbb{Z}}
\rho(k)^2<\infty$.  We have
\begin{equation}  \label{equ23}
  \sum_{ |k_j| \le n  \atop 1\le j \le M} |\rho( {\bf k}  \cdot  {\bf v} ) \rho( {\bf k}  \cdot  {\bf w} )|  \prod _{j=1} ^{M} | \rho(k_j) |   \le C \left(\sum_{|k| \leq n} |\rho(k)|\right)^{M-2}.
\end{equation}
where ${\bf k} = (k_1, \dots, k_M)$ and ${\bf v}, {\bf w} \in \R^M$ are  linearly independent vectors,  whose  components are  $0
$, $1$or $-1$ and they  have at least two nonzero components. 
\end{lemma}

\begin{proof}
Suppose first that $ \rho( {\bf k}  \cdot  {\bf v} ) \rho( {\bf k}  \cdot  {\bf w} )$ involves only three $k_i$'s, for instance, $k_1, k_2, k_3$. In this case, applying the Brascamp-Lieb inequality   (\ref{BL}) with exponents $p_i=3/5$, $1\le i\le 5$, yields,
\[
 \sum_{ |k_i| \le n  \atop 1\le i \le 3} |\rho(k_1)\rho(k_2)\rho(k_3)\rho( {\bf k}  \cdot  {\bf v} ) \rho( {\bf k}  \cdot  {\bf w} )|
 \le  \left( \sum_{|k|\le n} |\rho(k)| ^{\frac 53} \right)^3.
 \]
 Notice that  assumption (ii)  in Proposition \ref{prop2.10} is satisfied because  three of the vectors  $(1,0,0)$, $(0,1,0)$, $(0,0,1)$,
 ${\bf v}$, ${\bf w}$ may span a subspace of dimension $2$, and we have $3\times 3/5 =  9/5 \le 2$.
 Then, making the decomposition $|\rho(k)| ^{\frac 53}=|\rho(k)| ^{\frac 13}|\rho(k)| ^{\frac 43}$ and using H\"older's inequality with exponents $p=3$ and $q= \frac 32$, yields
\[
 \left( \sum_{|k|\le n} |\rho(k)| ^{\frac 53} \right)^3 \le  C  \sum_{|k|\le n} |\rho(k)|,
 \]
 which gives the desired estimate.
 
If $ \rho( {\bf k}  \cdot  {\bf v} ) \rho( {\bf k}  \cdot  {\bf w}  )$ involves four $k_i$'s, for instance, $k_1, k_2, k_3, k_4$, we  apply the Brascamp-Lieb inequality   (\ref{BL}) with exponents $p_i=2/3$, $1\le i\le 6$,  and we obtain
\[
 \sum_{ |k_i| \le n  \atop 1\le i \le 3} |\rho(k_1)\rho(k_2)\rho(k_3) \rho(k_4) \rho( {\bf k}  \cdot  {\bf v} ) \rho( {\bf k}  \cdot  {\bf w} )|
 \le  \left( \sum_{|k|\le n} |\rho(k)| ^{\frac 32} \right)^4.
 \]
 Then,  using (\ref{equ7}) with $M=2$, yields
\[
 \left( \sum_{|k|\le n} |\rho(k)| ^{\frac 32} \right)^4 \le  C  \left( \sum_{|k|\le n} |\rho(k)|\right)^2,
 \]
 which gives the desired estimate. Finally, if $ \rho( {\bf k}  \cdot  {\bf v} ) \rho( {\bf k}  \cdot  {\bf w} )$ involves more than four $k_i$'s, the result follows again from the Brascamp-Lieb inequality  (\ref{BL}), where we choose  $p_i=2/3$ for the factors
 $ \rho( {\bf k}  \cdot  {\bf v} ) $, $ \rho( {\bf k}  \cdot  {\bf  w})$ and  for the  four factors  $\rho(k_i)$ such that $k_i$ appears in the linear combination with less factors, and we choose $p_i=1$ for all the remaining factors   $\rho(k_i)$ appearing in the linear combinations  $ \rho( {\bf k}  \cdot  {\bf v} ) $ or $ \rho( {\bf k}  \cdot  {\bf w} )$.
\end{proof}

The last lemma summarizes some inequalities derived from the application of H\"older's inequality.
\begin{lemma}
For any $M\ge 2$, we have   
  \begin{equation}
\label{ho1}
		       \left(\sum_{|k| \leq n} |\rho(k)|^{1+\frac 1M} \right)^M  \le     \left( \sum_{|k| \le n} |\rho(k)|  \right) \left( \sum_{|k| \leq n} |\rho(k)|^{\frac  M{M-1}} \right)^{M-1}
		    \end{equation}
		    and
  \begin{equation}  
 \label{ho2}
 \left( \sum_{|k| \leq n} |\rho(k)|  \right)^3 \le  n   \left(\sum_{|k| \leq n} |\rho(k)|^\frac{3}{2} \right) ^{2}.
			   \end{equation}
Furthermore, if  $\sum_{|k| \leq n} |\rho(k)|^2 < \infty$, then
\begin{equation}
\label{ho3}
		       \sum_{|k| \leq n} |\rho(k)|^\frac{3}{2}  \leq C \left(\sum_{|k| \leq n} |\rho(k)|^\frac{4}{3}\right)^{\frac 34}.
		         \end{equation}
\end{lemma}

\begin{proof}
To show (\ref{ho1}) we  make use of the decomposition  $|\rho(k)|^{ 1+\frac 1M}= |\rho(k)| |\rho(k)|^\frac{1}{M}$ and apply H\"older's inequality with exponents $p= \frac  M {M-1}$ and $q=M$.  For   (\ref{ho3}) 
we use  the decomposition
$|\rho(k)|^\frac{3}{2}= |\rho(k)| |\rho(k)|^\frac{1}{2}$ and apply H\"older's inequality with exponents $p= \frac 43$ and $q=4$.    Finally, (\ref{ho2})  we use again H\"older's inequality. 
\end{proof}

\begin{minipage}{1.0\textwidth}
%\address{
\vskip 1cm
David Nualart and Hongjuan Zhou: Department of Mathematics, University of Kansas, 405 Snow Hall, Lawrence, Kansas, 66045, USA. 
%}
%\email{

{\it E-mail address:} nualart@ku.edu,  zhj@ku.edu
\end{minipage}
\end{document}